\documentclass[11pt]{amsart}
\usepackage{geometry}
\usepackage{tikz}
\usepackage[lite]{amsrefs}
\usepackage{amssymb}
\usepackage[all,cmtip]{xy}

\usepackage{graphicx}
\usepackage{array}
\usepackage{ tabularx}
\newcommand{\tensor}{\otimes}

\usepackage{amsmath}
\usepackage{xfrac} 
\usepackage{faktor}
\usepackage{tcolorbox}
\allowdisplaybreaks
\usepackage{submaxDE-macros}
\usepackage{times}
\usepackage[toc,page]{appendix}
\usepackage{multirow}
\usepackage[draft]{todonotes}
\usepackage{hyperref}
\usepackage[shortlabels]{enumitem}
\usepackage[normalem]{ulem}

\newcommand\ST{\operatorname{ST}}
\newcommand\Wilc{Wilczynski}
\newcommand\Cclass{C-class}

\date{\today}

\begin{document}

\title[Symmetry gaps for higher order ordinary differential equations]{ Symmetry gaps for higher order ordinary differential equations}

\author{Johnson Allen Kessy}
\address{Department of Mathematics and Statistics, 
	UiT The Arctic University of Norway, 9037 Troms\o, Norway}
\email{johnson.a.kessy@uit.no}

\author{Dennis The}
\address{Department of Mathematics and Statistics, 
UiT The Arctic University of Norway, 9037 Troms\o, Norway}
\email{dennis.the@uit.no}

\makeatletter
\@namedef{subjclassname@2020}{%
  \textup{2020} Mathematics Subject Classification}
\makeatother

\subjclass[2020]{Primary:
35B06, % Symmetries, invariants, etc. in context of PDEs
53A55; % Differential invariants (local theory), geometric objects
Secondary:
17B66, % Lie algebras of vector fields and related (super) algebras
57M60. % Group actions on manifolds and cell complexes in low dimensions
} 
\keywords{Submaximal symmetry for ODE, symmetry gap problem, Cartan geometry}

\begin{abstract} 
The maximal contact symmetry dimensions for scalar ODEs of order $\ge 4$ and vector ODEs of order $\ge 3$ are well known.  Using a Cartan-geometric approach, we determine for these ODEs the next largest realizable ({\sl submaximal}) symmetry dimension.  Moreover, finer curvature-constrained submaximal symmetry dimensions are also classified.
\end{abstract}
\maketitle

%%%%%%%%%%%%%%%%%%%%%%%%%%%%%%%%%%%%%%%%%%%%%%%%%%%%%%%%%
 \section{Introduction}
 \label{S:intro}
%%%%%%%%%%%%%%%%%%%%%%%%%%%%%%%%%%%%%%%%%%%%%%%%%%%%%%%%%

Consider a system of $m \geq 1$ ordinary differential equations (ODEs) of order $n+1 \geq 2$ given by 
\begin{align} \label{ODE:sym}
\bu^{(n+1)} = \mathbf{f}(t,\bu,\dot{\bu},\ldots,\bu^{(n)}),
\end{align}
where $\bu$ is an $\R^m$-valued function of $t$, and $\bu^{(k)}$ is its $k$-th derivative. We will focus on the geometry of such ODEs under local {\em contact} transformations, which by the Lie--B\"acklund theorem agrees with the geometry under local {\em point} transformations when $m \geq 2$ (vector ODEs).

 Except when $n=m=1$ (scalar 2nd order), the ODE \eqref{ODE:sym} admits a finite-dimensional contact symmetry algebra and the largest realizable ({\sl maximal}) symmetry dimension $\fM$ is known -- see for example \cite[\S 1]{BLP2021} for a historical survey.  Indeed, the trivial ODE $\bu^{(n+1)} = 0$ is uniquely (up to contact equivalence) maximally symmetric among \eqref{ODE:sym}, cf. Corollary \ref{cor:maxmodel} below, and the dimension of its Lie algebra of (infinitesimal) contact symmetries is given by
 \begin{align} \label{E:max}
 \fM = \begin{cases}
 10, & \mbox{if } m=1, \,n=2 \,\,\mbox{ (scalar 3rd order)};\\
 (m+2)^2-1, & \mbox{if } m\geq 2,\, n=1\,\, \mbox{ (vector 2nd order)};\\ 
 m^2+(n+1)m+3, & \mbox{if } m=1,\,n\geq 3\,\, \mbox{ or } \,\,m,n \geq 2\,\, \mbox{ (higher order cases)}.
 \end{cases}
 \end{align}
 In contrast, all scalar 2nd order ODEs are locally contact equivalent to the trivial ODE $\ddot{u} = 0$, which admits an {\em infinite}-dimensional contact symmetry algebra.  Under point transformations, $\ddot{u}=0$ has point symmetry algebra of dimension $\fM = 8$ and is maximally symmetric.

 In all cases with a finite maximal symmetry dimension, a natural classification problem is to {\em determine the next largest realizable ({\sl submaximal}) symmetry dimension $\fS$}.  There is often a sizable gap between $\fM$ and $\fS$, so this is referred to as the {\sl symmetry gap} problem.  For ODEs, examples of this are given in Table \ref{tab:parcases}.  See \cite{KT2017} for details on these cases where the underlying geometric structure is a {\sl parabolic geometry} (see below).
 
 \begin{table}[h] 
 \centering
 \[
 \begin{array}{|lccc|} \hline 
\multicolumn{1}{|c}{\mbox{Geometry}}& \fS &\mbox{Sample ODE} & \mbox{Reference}  \\ \hline \hline
\begin{array}{l}
\mbox{Scalar 2nd order ODEs} \\ \mbox{mod point transformations}
\end{array}	 &3 & \begin{array}{l}
\ddot{u} = \exp(\dot{u})
% \end{array}& \mbox{Tresse}\, (1896)\, \mbox{\cite{Tresse1896}}\\
 \end{array}& (1896)\, \mbox{\cite{Tresse1896}}\\
\begin{array}{l}
\mbox{Scalar 3rd order ODEs} \\ \mbox{mod contact transformations}
\end{array}	& 5& \begin{array}{l}
 \dddot{u}= b\dot{u} +u\\
%  \end{array}& \mbox{Wafo Soh, Mahomed, Qu}\, (2002)\, \mbox{\cite{WMQ2002}} \\
 \end{array}& (2002)\, \mbox{\cite{WMQ2002}} \\
\begin{array}{l}
\mbox{Vector 2nd order ODEs} \\ \mbox{mod point transformations}
\end{array}	 & m^2+5 & 
 \begin{array}{l}
 \underset{(1\le a \le m)}{\ddot{u}^a = (\dot{u}^1)^3 \delta^a_m}\\
 \end{array}& \begin{array}{l}
% m=2:\mbox{Casey, Dunajski, Tod}\, (2013)\, \mbox{\cite{CDT2013}}\\
 m=2: (2013)\, \mbox{\cite{CDT2013}}\\
 
% m \ge 3: \mbox{Kruglikov, The}\, (2017)\, \mbox{\cite{KT2017}}
 m \ge 3:  (2017)\, \mbox{\cite{KT2017}}
 \end{array} \\\hline
 	\end{array}
 	\]
 \caption{Submaximal symmetry dimensions $\fS$ for ODEs among parabolic geometries}
 \label{tab:parcases}
 \end{table}

 We consider the symmetry gap problem for higher order ODE (which are {\sl non-parabolic}), and prove that:
 \begin{thm} \label{T:main}
 Fix $(n,m)$ with $m=1,n \geq 3$ or $m,n\geq 2$.  Among the ODEs \eqref{ODE:sym} of order $n+1$, the submaximal contact symmetry dimension is
 \begin{align}
 \fS = \begin{cases}
 \fM - 1, & \mbox{if } m=1,\, n \in \{ 4, 6 \};\\
 \fM - 2, & \mbox{otherwise}.
 \end{cases}
 \end{align}
 \end{thm}
 
 This corrects a recent conjecture \cite[\S 10]{BLP2021} for $\fS$ when 
 $m,n \geq 2$, stated as 
$\begin{cases}
\fM -2m+2,& \mbox{if}\,\,m \in \{2,3\};\\
\fM -2m+1,& \mbox{if}\,\,m \ge 4.
\end{cases}$

The results for scalar ODEs recover Lie's \cite{Lie1924} (see \cite[p.205]{Olver1995} for a brief summary), which he obtained based on \cite[Thm.6.36]{Olver1995} and the complete classification of Lie algebras of contact vector fields on the (complex) plane.  This requires classifying the fundamental differential invariants for each such Lie algebra of vector fields as well as investigating their Lie determinants (see \cite[Table 5]{Olver1995}).  However, attempting to apply Lie's approach to vector ODEs in order to prove Theorem \ref{T:main} is not feasible: this would require as a first step classifying Lie algebras of vector fields in general dimension.  This is far out of reach, as evidenced by the fact that even the classification in dimension three remains incomplete (although large branches have been settled), see \cite{Doubrov2017, Schneider2018} for recent progress and references therein.  Moreover, even if such classifications were available, the computations involved with the approach would be extremely tedious, and establishing refinements as in Theorem \ref{T:main2} below would be even more difficult.  Different techniques are required to address the vector cases.

Our approach is based on a categorically equivalent reformulation of ODEs $\cE$ given by \eqref{ODE:sym} (mod contact) as {\sl regular, normal Cartan geometries} $(\cG \to \cE, \omega)$ of type $(G,P)$, for some appropriate Lie group $G$ and closed subgroup  $P \subset G$ (see \S \ref{S:TODE}).  The construction of such canonical Cartan connections $\omega$ for ODEs was discussed in \cite{DKM1999,Doubrov2001,DM2014,CDT2020}.  The trivial ODE corresponds to the flat model $(G \to \sfrac{G}{P},\omega_G)$, which has symmetry dimension $\dim G$, and more generally $\dim G$ bounds the symmetry dimension of any Cartan geometry of type $(G,P)$, so $\fM = \dim G$.

 Parabolic geometries are Cartan geometries modelled on the quotient of a semisimple Lie group by a parabolic subgroup.  For this diverse class of geometric structures (whose underlying structures includes those ODEs from Table \ref{tab:parcases}), significant progress on the symmetry gap problem was made in \cite{KT2017}.  In particular, a universal algebraic upper bound $\fU$ on $\fS$ was established, effective methods for the computation of $\fU$ were given in the complex or split-real settings, and in almost all of these cases it was shown that $\fS = \fU$ by presenting (abstract) models. 

 All higher order ODEs ($m=1,\,n\geq 3$ or $m,n \geq 2$) admit equivalent descriptions as {\sl non-parabolic} Cartan geometries.   For these ODEs, we adapt certain key features from the parabolic study to our specific non-parabolic setting. The main ingredients for establishing $\fS \leq \fU$ are {\sl harmonic curvature} $\kappa_H$, which is a complete obstruction to local flatness, and {\sl Tanaka prolongation}, both of which have parallels in the ODE setting.  The key technical fact underpinning our $\fS \leq \fU$ proof is that $\kappa_H \not\equiv 0$ is valued in a certain completely reducible $P$-module, which was established in \cite[Cor.3.8]{CDT2020}, so only the action of the reductive part $G_0 \subset P$ is relevant.  (In fact, the strategy of our proof is a simplified version of that given in \cite{KT2018}, which yields a stronger statement than the approach from \cite{KT2017} -- see Remark \ref{R-BT}.)  Our upper bound result is formulated in Theorem \ref{th:uub}.
 
 By complete reducibility, the codomain of $\kappa_H$ can be identified with a certain proper $G_0$-submodule $\bbE \subsetneq H^2_+(\g_-,\g)$ of a Lie algebra cohomology group.  This {\sl effective part} $\bbE$ has already been computed in the literature by Doubrov \cite{Doubrov2000,Doubrov2001} for scalar ODEs, Medvedev \cite{Medvedev2010} for vector 3rd order ODEs, and by Doubrov--Medvedev \cite{DM2014} for vector higher order ODEs.  In \S \ref{sect:main}, we summarize their classifications in Tables \ref{tab:scalar} and \ref{tab:vector}, organized as irreducible $G_0$-submodules $\bbU \subset \bbE$, and use these to efficiently compute the corresponding restricted quantities $\fU_\bbU$, from which $\fU$ can be obtained via \eqref{E:U-decomp}.
 
 We note that the aforementioned upper bound proof also yields the finer results $\fS_\bbU \leq \fU_\bbU$, where $\fS_\bbU$ is analogous to $\fS$ but with the additional constraint that $\kappa_H\not\equiv 0$ is valued in $\bbU \subset \bbE$.  Thus, we can consider the finer symmetry gap problem of determining $\fS_{\bbU}$ for a fixed $\bbU$.  For ODEs that are parabolic geometries, such constrained problems were resolved in \cite{KT2017}.  In our non-parabolic setting, using the known fundamental (relative) differential invariants for higher order ODEs derived in  \cite{Wilczynski1905, Se-ashi1988, Doubrov2001, Medvedev2011, DM2014}, we exhibit realizability of $\fU_{\bbU}$ in \S \ref{S:ODE-KH} by finding explicit ODEs realizing these symmetry dimensions and with $\kappa_H \neq 0$ concentrated in $\bbU$. In addition to proving Theorem \ref{T:main}, we obtain the following curvature-adapted result:

 \begin{thm} \label{T:main2}
 Fix $(n,m)$ with $m=1,n \geq 3$ or $m,n\geq 2$, and consider ODEs \eqref{ODE:sym} of order $n+1$.  Let $\bbU$ be a $G_0$-irrep contained in the effective part $ \bbE \subsetneq H_{+}^2(\g_{-},\g)$.  Then $\fS_{\bbU}$ is given in Table \ref{tab:T:main2}.
 \end{thm}
 
 \begin{table}[h] 
	\centering
	\begin{tikzpicture}
	\node[below left]{$
		\begin{array}{|cccl|} \hline 
		n &m & G_0 \mbox{-irrep }\, \bbU \subset \bbE & \mathfrak{S}_{\bbU}\\ 
		\hline\hline
		\geq 3 & 1 & \underset{(3 \le r \le n+1)}{\bbW_r} & \fM-2 = \fU_{\bbW_r}\\ \hline
		3& 1&\bbB_3&\fM -3 = \fU_{\bbB_3}-1 \\
		3& 1&\bbB_4&\fM -2=\fU_{\bbB_4} \\
		4& 1&\bbB_6&\fM-1 =\fU_{\bbB_6}\\
		\ge 4&1 &\bbA_2&\fM-2=\fU_{\bbA_2}\\
		5& 1&\bbA_3& \fM-3= \fU_{\bbA_3}-1\\
		\ge6& 1&\bbA_3& \leq \fM-3= \fU_{\bbA_3}-1\\
		6&1 &\bbA_4&\fM -1= \fU_{\bbA_4}\\
		\geq 7 &1 &\bbA_4& \fM -3 = \fU_{\bbA_4}-1 \mbox{ or } \fM - 4 \\ 	\hline
		\end{array}$};
	%%%%%%%%%%%%%%%%%%%%%
	\node[below right]{$
		\begin{array}{|cccl|} \hline 
		n &m &G_0\mbox{-irrep }\, \bbU \subset \bbE  &\fS_{\bbU} = \fU_{\bbU}\\ 
		\hline\hline	
		\ge 2 &\ge 2  &\underset{(2 \le r \le n+1)}{\bbW_r^{\tf}} & \fM -2m+1  \\
		\geq 2&\ge 2 & \underset{(3 \le r \le n+1)}{\bbW_r^{\tr}} & \fM-2 \\ \hline
		%%%%%%
		2 &\ge 2 &\bbB_4  &\fM -m \\ 
		2& \ge 2 & \bbA_2^{\tf}  &\fM -2m+2 \\ 	
		\ge 2& \ge 2 & \bbA_2^{\tf}  &\fM -2m+1 \\  		
		\geq 3& \ge 2& \bbA_2^{\tr}   &\fM -m -1 \\ \hline
		\end{array}$};
	\end{tikzpicture}
	
	(Recall $\fM = m^2+(n+1)m+3$ from \eqref{E:max}.)\\[0.1in]
	\caption{Curvature-constrained submaximal symmetry dimensions for ODEs of order $n+1$ } 
	\label{tab:T:main2}
\end{table}

 We note that all vector cases and most scalar cases satisfy $\fS_{\bbU} = \fU_{\bbU}$.  The exceptional scalar cases are: $(n, \bbU) = (3, \bbB_3), (\ge 5, \bbA_3)$ or $(\ge 7, \bbA_4)$.  The assertions $\fS_{\bbU} < \fU_{\bbU}$ here can be deduced from the known classification of submaximally symmetric scalar ODEs (see \cite[p. 206]{Olver1995}).  In Appendix \ref{S-Exc}, we outline an alternative algebraic method for establishing these $\fS_{\bbU} < \fU_{\bbU}$ exceptions.
 
 We conclude this introduction with explicit examples of ODEs (in Tables  \ref{tab:curvscalar} and \ref{tab:curvODE}) that realize $\fS_{\bbU}$ from Table \ref{tab:T:main2} (aside from the above exceptions).  We use the notation $\bu^{(k)} := (u_k^1,\ldots,u_k^m)$ for the $k$-th derivative of $\bu :=(u^1,\ldots,u^m)$ with respect to $t$.  The assertions about the given ODEs  can be directly verified using the relative invariants summarized in \S \ref{S:ODE-KH}  and explicit infinitesimal symmetries given in Tables \ref{tab:linearmodels},  \ref{tab:submax1}, and \ref{tab:submax-Kh}.

\begin{table} [h]
\centering
\[
\begin{array}{|ccc|} 	
\hline 
n &G_0 \mbox{-irrep }\, \bbU \subset \bbE &\mbox{Example ODE with}\, \img(\kappa_H) \subset \bbU \\ 
\hline	\hline
\ge 3 & \underset{(3 \leq r \leq n+1)}{\bbW_r} 
& u_{n+1} = u_{n+1-r} \\ \hline
3 & \bbB_4 & \multirow{2}{*}{$nu_{n-1}u_{n+1}-(n+1)(u_{n})^2 = 0$} \\
\ge 4&\bbA_2&\\
\hline
4 & \bbB_6 & 9 (u_2)^2 u_5 - 45 u_2 u_3 u_4 + 40(u_3)^3 = 0\\ \hline
6 &\bbA_4 & \begin{array}{c}
10(u_{3})^3u_{7}-70(u_{3})^2u_{4}u_{6} - 49(u_{3})^2(u_{5})^2\\
+ 280u_{3}(u_{4})^2u_{5} -175(u_{4})^4 = 0
\end{array} \\
\hline
\end{array}
\]
\caption{Scalar ODEs of order $n+1 \geq 4$ realizing $\fS_\bbU$}
\label{tab:curvscalar}
\end{table}

 \begin{table} [h]
 \centering
 \[
 \begin{array}{|ccc|} \hline
 n&G_0 \mbox{-irrep }\, \bbU \subset \bbE & \mbox{Example ODE with } \img(\kappa_H) \subset \bbU \\\hline \hline 
 \geq 2  & \underset{(3 \leq r \leq n+1)}{\bbW_r^{\tr}} & \underset{(1 \leq a \leq m)}{u_{n+1}^a = u_{n+1-r}^a}\\ \hline	 
 \ge 2 & \underset{(2 \leq r \leq n+1)}{\bbW_r^{\tf}} &
 \underset{(1 \leq a \leq m)}{u_{n+1}^a = u_{n+1-r}^2 \delta_1^a}
  \\ \hline
	2 &\bbB_4  & \multirow{3}{*}{$\begin{array}{c}
		\underset{(1 \le a \le m)}{u_{n+1}^a= \displaystyle\frac{(n+1)u_n^1 u_n^a}{nu^1_{n-1}}}	\\
		\end{array}$}\\[0.5cm]
	\ge 3 &\bbA_2^{\tr}& \\
	\hline			
	\ge 2  &\bbA_2^{\tf}  & 
	\begin{array}{c}
	\underset{(1 \le a \le m)}{u_{n+1}^a= (u_n^2)^2\delta_1^a}	\\
	\end{array} \\
	\hline
	\end{array}
	\]
	\caption{Vector ODEs of order $n+1 \geq 3$ (for $m \geq 2$ functions) realizing $\fS_\bbU $}
	\label{tab:curvODE}
\end{table}

%%%%%%%%%%%%%%%%%%%%%%%%%%%%%%%%%%%%%%%%%%%%%%%%%%%%%%%%%
 \section{An upper bound  on submaximal symmetry dimensions}
%%%%%%%%%%%%%%%%%%%%%%%%%%%%%%%%%%%%%%%%%%%%%%%%%%%%%%%%%

 We begin by reviewing the Cartan-geometric perspective on ODEs, and then use it to prove an upper bound formula for submaximal symmetry dimensions (Theorem \ref{th:uub}).
 
%%%%%%%%%%%%%%%%%%%%%%%%%%%%%%%%%%%%%%%%%%%%%%%%%%%%%%%%%
\subsection{Canonical Cartan connections}
%%%%%%%%%%%%%%%%%%%%%%%%%%%%%%%%%%%%%%%%%%%%%%%%%%%%%%%%%
\subsubsection{ODEs as filtered $G_0$-structures} 
\label{subsec:ODEfilt}
%%%%%%%%%%%%%%%%%%%%%%%%%%%%%%%%%%%%%%%%%%%%%%%%%%%%%%%%%

Consider the space $J^{n+1}(\R, \R^m)$ of $(n+1)$-jets of  smooth maps from $\R$ into $\R^m$, with  the natural projection $\pi_n^{n+1} : J^{n+1}(\R, \R^m) \to J^{n}(\R, \R^m)$ and  denote by $C$ the {\sl Cartan distribution} on it.   Denoting $\bu_r = (u_r^1,\ldots,u_r^m)$, we let $(t, \bu_0, \bu_1,\ldots,\bu_{n+1})$ be standard (bundle-adapted) local coordinates on $J^{n+1}(\R, \R^m)$, for which the Cartan distribution $C$ is given by
\begin{align}
C = \langle \partial_t + \bu_1 \partial_{\bu_0} + \ldots + \bu_{n+1} \partial_{\bu_{n}}, \partial_{\bu_{n+1}} \rangle.
\end{align}
 (Here, $\bu_1 \partial_{\bu_0}$ is our compact notation for $\sum_{a=1}^m u_1^a \partial_{u_0^a}$, etc. and $\partial_{\bu_{n+1}}$ refers to $\partial_{u_{n+1}^1},\ldots,\partial_{u_{n+1}^m}$.)

We will consider \eqref{ODE:sym} up to {\sl contact transformations}. These are diffeomorphisms $\phi$ of $J^{n+1}(\R, \R^m)$ that preserve the distribution $C$, i.e.\ $\phi_{\ast} (C) = C$. By the Lie--B\"{a}cklund theorem, such transformations are the {\sl prolongations} \cite{Olver1995} of contact transformations on $J^1(\R, \R^m)$.  Moreover, for $m \ge 2$ they are the prolongations of diffeomorphisms on $J^0(\R, \R^m) \cong \R \times \R^m$ ({\sl point transformations}). At the infinitesimal level, a {\sl contact vector field} $\xi$ is a vector field whose flow is a (local) contact transformation.  Equivalently, $\cL_\xi C \subset C$, where $\cL_\xi$ is the Lie derivative with respect to $\xi$.

 Rephrased geometrically, the $(n+1)$-st order ODE \eqref{ODE:sym} is a codimension $m$ submanifold $\mathcal{E} = \{ \bu_{n+1} = \mathbf{f} \}$ in  $J^{n+1}(\R, \R^m)$ transverse to the projection map $\pi_n^{n+1}$. So,  $\cE$ can be (locally) identified  with its diffeomorphic image in $J^n(\R, \R^m)$.
\begin{defn}\label{Symm}
A {\sl contact symmetry} of the ODE $\cE \subset J^{n+1}(\R,\R^m)$ is a contact vector field $\xi$ on $J^{n+1}(\R, \R^m)$ that is tangent to $\cE$.
\end{defn}  

We associate $\cE$  with a pair $(E , V)$ of subdistributions of $C$ described below:
\begin{itemize}
	\item the line bundle $E$ over $\cE$ whose integral curves are lifts of solution curves to \eqref{ODE:sym};
	\item the rank $m$ Frobenius-integrable distribution $V := \ker(d\pi^{n+1}_n|_\cE)$.
\end{itemize}
As proven in \cite[Thm 1]{DKM1999}, the pair $(E, V)$ encodes $\cE$ up to the contact transformations and therefore defines a geometric structure associated to \eqref{ODE:sym}.

 Equivalently, a contact symmetry of the ODE $\cE \subset J^{n+1}(\R,\R^m)$ is a  vector field $\xi$ on $\cE$  such that $\cL_\xi E \subset E$ and $\cL_\xi V \subset V$.
In standard local coordinates,
\begin{align} \label{E:EV}
E = \left\langle \tfrac{d}{dt} :=\partial_t + \mathbf{u}_1\partial_{\mathbf{u}_0} + \cdots + \mathbf{u}_{n}\partial_{\mathbf{u}_{n-1}} + \mathbf{f}\partial_{\mathbf{u}_{n}}\right\rangle, \quad
V= \left \langle \partial_{\mathbf{u}_n} \right \rangle.
\end{align}
In the sequel, we shall refer to $\frac{d}{dt}$ as the {\sl total derivative}.

The distribution $D := E \oplus V \subset T\mathcal{E}$ is bracket-generating and its weak-derived flag defines a filtration on the tangent bundle $T\cE$: 
\begin{align}\label{F-ODE}
T\mathcal{E} = D^{-n-1} \supset \cdots \supset D^{-2} \supset D^{-1},
\end{align}
where $D^{-1} := D$ and $D^{-j-1} := D^{-j} + [D^{-j}, D^{-1}]$ 
for $j > 0$. Then $(\mathcal{E}, \{D^j\})$  becomes a \emph{filtered manifold}, since the Lie bracket of vector fields on $\cE$ is compatible with the tangential filtration $\{D^j\}$, i.e
\begin{align}
 [\Gamma(D^i), \Gamma(D^j)] \subset \Gamma(D^{i+j}). 
\end{align}
 From \eqref{E:EV}, we can moreover verify that
\begin{align}\label{c:SR}
[\Gamma(D^i), \Gamma(D^j)] \subset \Gamma(D^{\min(i,j)-1}), 
\end{align}
which is a stronger condition if $i,j \le -2$.

Furthermore, \eqref{ODE:sym} admits an equivalent description as a {\sl filtered $G_0$-structure} described below.
The {\sl associated graded} to the filtration \eqref{F-ODE} is given by
\begin{equation*}
\gr(T \mathcal{E}) := \bigoplus^{-1}_{j=-n-1} \gr_j(T\mathcal{E}),\quad \text{where}\quad \gr_j(T\mathcal{E}) :=  D^j\mathcal{E} / D^{j+1}\mathcal{E}.
\end{equation*}
For $x \in \cE$, the Lie bracket of vector fields induces a (Levi) bracket on $\fm (x):=\gr(T_x \cE )$  turning it into a  nilpotent graded Lie algebra (NGLA) with $\fm_j(x):= \gr_j(T_x\cE)$. It is called the {\sl symbol algebra} at $x$. For distinct points $x, y \in\cE $, $\mathfrak{m}(x)$ and $\fm (y)$ belong to the same NGLA isomorphism class.  Let $\fm$ be a fixed NGLA with $\mathfrak{m} \cong \fm (x), \forall x \in \cE$. Since $D$ is bracket-generating, then $\fm$ is generated by $\fm_{-1}$. 

For $x \in \cE$, denote by $F_{\gr}(x)$ the set of all NGLA isomorphisms from $\mathfrak{m}$ to $\mathfrak{m}(x)$ and $F_{\gr}(\cE) := \bigcup_{x \in \cE} F_{\gr}(x)$. Then $F_{\gr}(\cE) \to \cE$ is a principal fiber bundle with structure group $\aut_{\gr}(\fm)$ consisting of all graded automorphisms of $\fm$. In fact, $\aut_{\gr}(\mathfrak{m}) \hookrightarrow\GL(\fm_{-1})$, since $\fm $ is generated by $\fm_{-1}$.

The splitting of $D$ implies a splitting of $\fm_{-1}$.  Let $G_0 \leq \aut_{\gr}(\fm)$ be the subgroup preserving this splitting of $\fm_{-1}$.  There is a corresponding proper subbundle $\cG_0 \to \cE$, which is a principal fiber bundle with reduced structure group  $G_0 \cong \R^{\times} \times \GL_m$.  This realizes the ODE as a so-called {\sl filtered $G_0$-structure} \cite[Defn 2.2]{Andreas2017}. We immediately caution that not all filtered $G_0$-structures arise from ODEs (see Remark \ref{R-G0}).

%%%%%%%%%%%%%%%%%%%%%%%%%%%%%%%%%%%%%%%%%%%%%%%%%%%%%%%%%
 \subsubsection{The trivial ODE}
 \label{S:TODE}
%%%%%%%%%%%%%%%%%%%%%%%%%%%%%%%%%%%%%%%%%%%%%%%%%%%%%%%%%
 
 Consider the trivial system of $m \ge 1$ ODEs $\bu_{n+1} = 0$ of order $n+1$.  Throughout, we will restrict to the {\sl higher order} cases $m=1,n\geq 3$ and $m,n\geq2$.  The contact symmetry vector fields for the trivial ODE were given in \cite[Section 2.2]{CDT2020}.  Abstractly, the contact symmetry algebra $\g$ has the structure
 \begin{align}
 \g := \mathfrak{q} \ltimes V,\; \quad\mbox{where}\quad \mathfrak{q}:=\fsl_2 \times \mathfrak{gl}_m, \quad V := \V_n \tensor W.
 \end{align}
 Here, $\V_n$ is the unique (up to isomorphism) $\fsl_2$-irrep of dimension $n+1$ and $W = \R^m$ is the standard representation of $\mathfrak{gl}_m$.  The trivial ODE admits the maximal symmetry dimension among \eqref{ODE:sym} for fixed $(n,m)$, c.f.\ Corollary \ref{cor:maxmodel}.  Consequently, we denote:
 \begin{align} \label{E:maxsym}
 \fM := \dm \g = m^2 + (n+1)m + 3.
 \end{align}
 
 We work with the following basis for $\g$.  Let $\{ w_a \}$ be the standard basis for $W = \R^m$, let $\gl_m \cong  \gl(W)$ be spanned by $\{ e^a_b \}$, where $e^a_b w_c = \delta^a_c w_b$, and let $\id_m := \sum_{a=1}^m e^a_a$.  Letting $\{x,y\}$ be the standard basis for $\R^2$, consider the standard $\fsl_2$-triple
 \begin{align}
 \sfX = x\partial_y,\quad \sfH = x\partial_x - y\partial_y, \quad \sfY = y\partial_x,
 \end{align}
 and consider the weight vectors for $\V_n$ given by
 \begin{align}
 E_{i} = \frac{1}{i!}x^{n-i}y^{i}, \quad i=0, \ldots, n.
 \end{align}

 Following \cite{Doubrov2001, DM2014}, we give $\g$ the structure of a $\mathbb{Z}$-graded Lie algebra $\g = \g_{-n-1} \oplus \ldots \oplus \g_1$, where
 \begin{align} \label{E:grading}
 \begin{split}
 \g_{1} &=  \R \sfY, \quad\g_{0} =  \R \sfH \oplus \gl_m, \quad
 \g_{-1} =  \R \sfX \oplus (\R E_n \tensor W),\\
 \g_{i} &= \R E_{n+1+i} \tensor W, \quad i = -2,\ldots, -n-1.
 \end{split}
 \end{align}
We note that $\g_{-} \cong \fm$, the symbol algebra defined in \S  \ref{subsec:ODEfilt}.

 The splitting on $\g_{-1}$ reflects the splitting on the distribution $D = E \oplus V$ from \S \ref{subsec:ODEfilt}. 
 Note that $\g_0$ is reductive and $\g_-$ is generated by $\g_{-1}$.  Alternatively, introducing the {\sl grading element} 
 \begin{align} \label{E:gr-elt}
 \sfZ := -\frac{1}{2}\left(\sfH + (n+2)\id_m \right),
 \end{align}
 the eigenspaces of $\ad_{\sfZ} \in \gl(\g)$ are precisely $\g_i = \{ x \in \g : [\sfZ,x] = i x \}$ for all $i \in \Z$.  We visualize this as in Figure \ref{F:grading}.

  \begin{figure}[h]
 \[
 \begin{tikzpicture}[scale=2,baseline=-3pt]
 \draw (0.05,-1) ++(0,0) node {$\cdots$}; 
 \draw[thick,red] (0,0) -- (1,0);
 \draw[thick] (0,0) -- (-1,0);
 \draw[thick] (0,0) -- (-1.5,-1);
 \draw[thick] (0,0) -- (-0.5,-1);
 \draw[thick] (0,0) -- (0.5,-1);
 \draw[thick] (0,0) -- (1.5,-1);
 \filldraw[red] (1,0) circle (0.05);
 \filldraw[red] (0,0) circle (0.05);
 \filldraw[black] (1.5,-1) circle (0.05);
 \filldraw[black] (0.5,-1) circle (0.05);
 \filldraw[black] (-0.5,-1) circle (0.05);
 \filldraw[black] (-1.5,-1) circle (0.05);
 \filldraw[black] (-1,0) circle (0.05);
 \draw (-1,0) ++(0,-0.2) node {\tiny $-1$};
 \draw (0,0) ++(0,-0.2) node {\tiny $0$};
 \draw (1,0) ++(0,-0.2) node {\tiny $1$};
 \draw (1.5,-1) ++(0,-0.2) node {\tiny $-1$};
 \draw (0.5,-1) ++(0,-0.2) node {\tiny $-2$};
 \draw (-0.5,-1) ++(0,-0.2) node {\tiny $-n$};
 \draw (-1.5,-1) ++(0,-0.2) node {\tiny $-n-1$};
 \draw (-1,0) ++(0,0.2) node {$\sfX$};
 \draw (0,0) ++(0,0.2) node {$\sfH, \id_m$};
 \draw (1,0) ++(0,0.2) node {$\sfY$};
 \draw (-1.5,-1) ++(0,0.2) node {$E_0$};
 \draw (-0.5,-1) ++(-0.1,0.2) node {$E_1$};
 \draw (0.5,-1) ++(0.15,0.2) node {$E_{n-1}$};
 \draw (1.5,-1) ++(0,0.2) node {$E_n$};
 \draw [draw=black,dashed] (-0.1,-0.1) rectangle (1.1,0.1);
 \end{tikzpicture}
 \]
 \caption{Grading on $\g$, with basis specified in the scalar case}
 \label{F:grading}
 \end{figure}
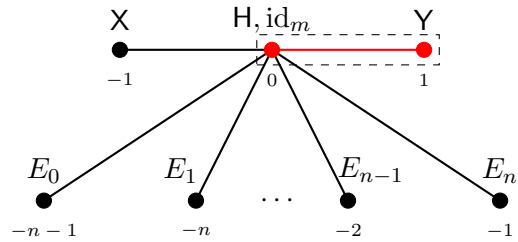

 We also endow $\g$ with the corresponding filtration $\g^{i} := \sum_{j \ge i}\g_{j}$, and let 
 \begin{align}
 \mathfrak{p}:= \g^{0} = \langle \sfH, e^a_b, \sfY \rangle, \quad \mathfrak{p}_{+}:= \g^{1} =\langle \sfY \rangle.
 \end{align}
 Let $\gr_i : \g^i \to \g^i / \g^{i+1}$ denote the natural quotient and let $\gr(\g) := \bigoplus_i \gr_i(\g)$ denote the associated graded, which is isomorphic as a $\g_0 \cong \gr_0(\g)$ module to $\g$ as a graded Lie algebra.
 
 At the group level, let
 \begin{itemize}
 \item $m=1$: $G = \GL_2 \ltimes \mathbb{V}_n$ and $P = \ST_2 \subset \GL_2$, the subgroup of lower triangular matrices;
 \item $m \ge 2$: $G = (\SL_2 \times \GL_m) \ltimes V$ and $P = \ST_2 \times \GL_m $.
 \end{itemize}
 In either case, let $G_0 := \{ g \in P : \Ad_g (\g_0) \subset \g_0 \}$.  We note that the filtration on $\fg$ is $P$-invariant.

%%%%%%%%%%%%%%%%%%%%%%%%%%%%%%%%%%%%%%%%%%%%%%%%%%%%%%%%%
 \subsubsection{Cartan geometries}
 \label{subsec:Cartangeom}
%%%%%%%%%%%%%%%%%%%%%%%%%%%%%%%%%%%%%%%%%%%%%%%%%%%%%%%%%

 All ODEs \eqref{ODE:sym} are filtered $G_0$-structures, and these admit an equivalent description as (normalized) Cartan geometries of type $(G,P)$.  We describe the precise setup in this section.
  
 \begin{defn}
 A {\sl Cartan geometry} $(\cG \to M, \omega)$ of type $(G,P)$ consists of a (right) principal $P$-bundle $\cG \to M$ endowed with a $\g$-valued one-form $\omega \in \Omega^1(\cG,\g)$, called a \emph{Cartan connection}, such that:
\begin{enumerate}
	\item[(i)] For any $u \in \cG$, $\omega_u : T_u \cG \to \fg$ is a linear isomorphism;
	\item[(ii)] $\omega$ is $P$-equivariant, i.e.\ $R_g^{\ast}\omega = \text{Ad}_{g^{-1}} \circ \omega$ for any $g \in P$;
	\item[(iii)] $\omega(\zeta_A) = A$, where $A \in \mathfrak{p}$, where $\zeta_A$ is the {\sl fundamental vertical vector field} defined by $\zeta_A(u) := \frac{d}{dt}{\big|}_{t=0} \,u\cdot\exp(tA)$.
\end{enumerate}
\end{defn}

 Because of (i), the tangent bundle of $\cG$ is trivialized, i.e.\ $T \cG \cong \cG \times \fg$, and the $P$-invariant filtration on $\fg$ induces a corresponding filtration of $T\cG$:
 \begin{align}
 T^{-n-1} \cG \supset\ldots \supset T^{-1} \cG \supset T^0 \cG \supset T^1 \cG.
 \end{align}
 Let us also note the following consequence of (ii).  Fixing $u \in \cG$, consider a $P$-invariant vector field $\eta \in \Gamma(T\cG)^P$ with $A:= \omega(\eta_u) \in \mathfrak{p}$, and let $f$ be a $P$-equivariant function on $\cG$.  Then:
 \begin{align} \label{E:Pequiv}
	(\eta \cdot f)(u) &= \left.\frac{d}{dt} \right|_{t=0}\,f\left(u\cdot \exp(A t)\right) =  \left.\frac{d}{dt}\right|_{t=0}\, \exp(-At) \cdot f(u) =-A\cdot f(u).
 \end{align}

 The {\sl Klein geometry} $(G \to \sfrac{G}{P}, \omega_G)$, where $\omega_G$ is the \emph{Maurer--Cartan form} on $G$, is called the \emph{flat model} for Cartan geometries of type $(G,P)$.  Given a Cartan geometry, its {\sl curvature form} $K \in \Omega^2(\cG,\g)$ is given by 
 \begin{align}
 K(\xi, \eta) = d\omega(\xi, \eta) + [\omega(\xi),\omega(\eta)],
 \end{align}
 which is $P$-equivariant and horizontal, i.e.\ $K(\zeta_A,\cdot) = 0, A \in \mathfrak{p}$.  By horizontality, it is determined by the $P$-equivariant {\sl curvature function} $\kappa : \cG \to  \bigwedge^2(\sfrac{\g}{\fp})^{\ast}\tensor \g$, defined by 
 \begin{align}
 \kappa (A, B) = K(\omega^{-1}(A),\omega^{-1}(B)), \quad A,B \in \g.
 \end{align}
 For $(G,P)$ from \S \ref{S:TODE}, and the filtration $\{ \g^i \}$ introduced there, we say that a Cartan connection $\omega$ is {\sl regular} if $\kappa (\g^i, \g^j) \subset \g^{i+j+1}$ for all $i,j$.  Equivalently, $\kappa$ has image in the subspace of $\bigwedge^2(\sfrac{\g}{\fp})^{\ast}\tensor \g$ on which the grading element $\sfZ$ acts with positive eigenvalues ({\sl degrees}).
 
  For normality of $\omega$, we follow the description in \cite[\S 3]{CDT2020}. Let us denote by $C^k(\g,\g) :=\bigwedge^2\g^{\ast} \tensor \g$, and consider the $P$-invariant subspace
 \begin{align}
 C_{\hor}^k(\g,\g):=\{\psi \in C^k(\g,\g):\iota_A\psi = 0, \forall A \in \mathfrak{p}\} \cong \bigwedge{}^{\!\!k} (\sfrac{\g}{\fp})^* \otimes \g.
 \end{align}
 Both of these inherit filtrations from the filtration on $\g$.  Their associated graded can be identified with $C^k(\g_{-},\g)$, i.e.\ the cochain spaces for a complex $C^\bullet(\g_{-},\g)$ with the standard differential $\partial$ for computing Lie algebra cohomology groups $H^k(\g_{-},\g)$.  There is an inner product $\langle \cdot, \cdot \rangle$ on $\g$ whose extension to $C^k(\g,\g)$ is such that the adjoint $\partial^{\ast}$ of the standard differential $\partial_{\g}$ on $C^\bullet(\g,\g)$ (with respect to $\langle \cdot, \cdot \rangle$) restricts to a $P$-equivariant map $\partial^{\ast} :\bigwedge^k(\sfrac{\g}{\mathfrak{p}})^{\ast} \tensor \g \to \bigwedge^{k-1}(\sfrac{\g}{\mathfrak{p}})^{\ast} \tensor \g$.  (See \cite[Lemma 3.2]{CDT2020} for details.)  In terms of this map $\partial^*$, we say that $\omega$ is  {\sl normal} if $\partial^{\ast}\kappa = 0$.  From \cite[Thm.2.2]{CDT2020} (see also \cite{DKM1999,Doubrov2001,DM2014}), we have the following important starting point:
 
 \begin{thm} Fix $(G,P)$ as above.  There is an equivalence of categories between filtered $G_0$-structures and regular, normal Cartan geometries of type $(G,P)$.
 \end{thm}
  
  \begin{rem}\label{R-G0}
  A  regular, normal Cartan connection {\sl associated to an ODE} \eqref{ODE:sym} satisfies the {\sl strong regularity} condition 
  $\kappa (\g^i, \g^j) \subset \g^{i+j+1} \cap \g^{\min(i,j)-1},\, \forall i,j $  \cite[Rem 2.3]{CDT2020}.
 Consequently, not all filtered $G_0$-structures arise from ODE.  For example, in \cite[\S 3.5]{CDT2020} there is a $G_2$-invariant filtered $G_0$-structure with the same symbol as that of an 11th order scalar ODE, but it is not realizable by any such ODE.
  \end{rem} 
 Since $(\partial^{\ast})^2 = 0$, then for regular, normal Cartan geometries one obtains the ($P$-equivariant) {\sl harmonic curvature} function 
 \begin{align}
 \kappa_H : \cG \to \frac{\ker\partial^{\ast}}{\img \partial^{\ast}},
 \end{align}
 which is valued in the filtrand of positive degree (by regularity).  It is a fundamental fact that $\kappa_H$ completely obstructs local flatness \cite{Andreas2017}, i.e $\kappa_H \equiv 0$ if and only if the geometry is locally equivalent to the flat model, which corresponds to the trivial ODE.  Furthermore,
 
 \begin{lem} \label{lem:compreducibility}
 The $P$-module $\frac{\ker\partial^{\ast}}{\img \partial^{\ast}}$ is completely reducible, i.e. $\g^1$ acts trivially.  
 \end{lem}
 
 \begin{proof}
 See \cite[Corollary 3.8]{CDT2020}.
 \end{proof}

 The above complete reducibility property will be important in subsequent sections.  Consequently, only the $G_0$-action on $\frac{\ker\partial^{\ast}}{\img \partial^{\ast}}$ is relevant. 
 Identifying $\bigwedge^2 (\sfrac{\g}{\fp})^* \otimes \g \cong \bigwedge^2 \g_-^* \otimes \g$ as $G_0$-modules, and defining the Laplacian operator $\square:=\partial \circ \partial^{\ast} + \partial^{\ast}\circ\partial$ on $\bigwedge^2 \g_-^* \otimes \g$, we have a Hodge decomposition and the following $G_0$ isomorphisms:
 \begin{align}\label{HD}
 \bigwedge{}^{\!\!2}{\g_{-}^{\ast}\tensor \g} \cong
 \rlap{$\overbrace{\phantom{\,\img \partial^{\ast} \oplus \ker\square}}^{\ker \partial^{\ast}}$}\img \partial^{\ast} \oplus  
 \underbrace{\ker \square \oplus \img \partial}_{\ker \partial},\quad \ker \square \cong \frac{\ker \partial^{\ast}}{\img \partial^{\ast}}\cong \frac{\ker \partial}{\img \partial}=:H^2(\g_{-},\g).
 \end{align}
 Regularity of $\omega$ and complete reducibility imply that the codomain of $\kappa_H$ can be identified with the subspace $H^2_{+}(\g_{-},\g) \subset H^2(\fg_-,\fg)$ on which $\sfZ$ acts with positive eigenvalues.

 Not all filtered $G_0$-structures are realizable by ODE, so some of $H^2_+(\g_-,\g)$ is extraneous for ODE.
 
 \begin{defn} \label{D-EP}
 Let $\bbE \subset H^2_+(\g_-,\g)$ denote the {\sl effective part}, i.e.\ the minimal $G_0$-submodule in which $\kappa_H$ is valued, for any regular, normal Cartan geometry of type $(G,P)$ associated to an ODE (for fixed $n,m$).
 \end{defn}
  This important submodule has already been computed in the literature \cite{Doubrov2000, Doubrov2001, Medvedev2010, DM2014}.  All irreducible components are summarized in Tables \ref{tab:scalar} and \ref{tab:vector}.

%%%%%%%%%%%%%%%%%%%%%%%%%%%%%%%%%%%%%%%%%%%%%%%%%%%%%%%%%
 \subsection{ODE symmetries viewed Cartan-geometrically}
%%%%%%%%%%%%%%%%%%%%%%%%%%%%%%%%%%%%%%%%%%%%%%%%%%%%%%%%%

 Given a Cartan geometry $(\cG \to M, \omega)$ of type $(G,P)$, an {\sl (infinitesimal) symmetry} is a $P$-invariant vector field on $\cG$ that preserves $\omega$ under Lie differentiation.  The collection of all such symmetries forms a Lie algebra, which we denote by
 \begin{align}
 \finf(\cG, \omega) := \left\{\xi \in \Gamma(\cG)^{P} :\mathcal{L}_{\xi}\omega = 0 \right\}.
 \end{align}
 
 \begin{prop} \label{prop:filtration}
	Let $(\cG \to M, \omega)$ be a Cartan geometry of type $(G, P)$ and fix $u \in \cG$ arbitrary.  Then:
	\begin{itemize}
 	\item [(i)] The map $\xi \mapsto \omega(\xi_u)$ is a linear injection from $\finf\,(\cG,\omega)$ into $\g$.  Let $\ff(u)$ denote the image subspace.
 	\item [(ii)] Equipping $\ff(u)$ with the inherited filtration $\ff(u)^k := \ff(u) \cap \g^k$ and bracket
 \begin{align}
 [X, Y]_{\ff(u)}:= [X, Y] - \kappa(u)(X, Y), \quad \forall X, Y \in \ff(u),
 \end{align}
 we have that $(\ff(u), [\cdot, \cdot]_{\ff(u)})$ is a filtered Lie algebra isomorphic to $\finf(\cG,\omega)$.
 	\item [(iii)]  The associated graded Lie algebra $\fs(u) := \gr(\ff(u))$ is a graded Lie subalgebra of $\g$. 
 	\item[(iv)] $\fs_0(u) \subseteq \mathfrak{ann}(\kappa_H(u)) \subseteq \g_0$.
	\end{itemize}	
 \end{prop}
 \begin{proof}
 The statements (i)--(iii) were proved in \cite[Thm.4]{CN2009} for bracket-generating distributions that lead to parabolic geometries of type $(G,P)$.  Although $(G,P)$ there refers to  the parabolic setting, the same proof works for our $(G,P)$ considered here.  For (iv), let $A \in \fp$ with $A \in \ff^0(u)$, and let $\eta$ be a symmetry with $\omega(\eta_u) = A$.  Use \eqref{E:Pequiv} with $f = \kappa_H$ to obtain $A \cdot \kappa_H(u) = 0$.  Since $\frac{\ker\partial^{\ast}}{\img \partial^{\ast}}$ is completely reducible, this statement only depends on $A\!\!\mod \ff^1 \in \fs_0(u)$, so (iv) follows.
 \end{proof}

 Using Cartan-geometric methods, we have:

 \begin{cor} \label{cor:maxmodel}
 Let $(n,m) \neq (1,1)$.  Up to (local) contact transformations, the trivial ODE $\bu^{(n+1)} = 0$ of order $n+1 \geq 2$ with $m \geq 1$ dependent variables is uniquely maximally symmetric among \eqref{ODE:sym}.
 
  %$m=1,n\geq 3$ or $m,n \geq 2$ is uniquely maximally symmetric.
 %$m=1,n\geq 2$ or $m,n \geq 2$ is uniquely maximally symmetric.
 \end{cor}

 \begin{proof} The scalar 3rd order ($n=2, m=1$) and vector 2nd order ($n=1, m\geq 2$) cases correspond to parabolic geometries -- see \cite[Prop.2.3.2]{KT2017} for a uniqueness statement.  The proof for higher order ODE cases is analogous and we give this here.  Given an ODE \eqref{ODE:sym}, let $(\cG \to M, \omega)$ be the corresponding regular, normal Cartan geometry of type $(G,P)$.  Fix any $u \in \cG$.  By Proposition \ref{prop:filtration} (iii), $\fs(u) \subset \fg$, so $\dim \inf(\cG,\omega) = \dm \fs(u) \leq \dm \g$.  The trivial ODE in particular has symmetry dimension $\fM = \dim \g$, so this is indeed maximal.  Now supposing $\dim \inf(\cG,\omega) = \dm \g$, we must have $\fs(u) = \g$, so $\g_0 = \fs_0(u) = \mathfrak{ann}(\kappa_H(u))$ follows from Proposition \ref{prop:filtration} (iv).  In particular, the grading element satisfies $\sfZ \in \fs_0(u)$.  Since $\kappa_H(u) \in H^2_{+}(\g_{-},\g)$, then $\kappa_H(u) = 0$, so $\kappa_H \equiv 0$ and the geometry is flat.  Thus, the ODE is locally equivalent to the trivial one.
 \end{proof}

 We note that the results for the scalar case are due to Lie \cite{Lie1893}, while Fels \cite{Fels1993} established uniqueness for the case of second and third order systems using Cartan's method of equivalence.
 
 %%%%%%%%%%%%%%%%%%%%%%%%%%%%%%%%%%%%%%%%
 \subsection{An algebraic bound on submaximal symmetry dimensions}
 \label{S:upper}
 %%%%%%%%%%%%%%%%%%%%%%%%%%%%%%%%%%%%%%%%
 
 Fix $(G,P)$ as above.  We define the submaximal symmetry dimension $\mathfrak{S}$ by:
 \begin{align}\label{D-submax}
 \begin{split}
 \mathfrak{S} &:= \max \left\{\dm \finf\,(\cG,\omega) : (\cG \to M, \omega) \, \mbox{regular, normal of type } (G,P) \right.\\
 & \hspace{2in}\left. \mbox{ associated to an ODE, with }\, \kappa_H \not\equiv 0 \right\}.
 \end{split}
 \end{align}
 
 Following \cite{KT2017}, we define:
 
 \begin{defn} \label{Tanakaalgebra}
 Let $\g$  be a graded Lie algebra with $\g_{-}$ generated by $\g_{-1}$. For $\fa_0 \subset \g_0$, the \emph{Tanaka prolongation algebra} is the graded subalgebra $\fa :=\text{pr}(\g_{-},\fa_0)$ of $\g$ with $\fa_{-} :=\g_{-}$ and $\fa_k$ defined iteratively for $k > 0$ by $\fa_k := \{X \in \g_k : [X, \g_{-1}] \subset \fa_{k-1}\}$.  Given $\phi$ in some $\g_0$-module, let $\fann(\phi) \subset \g_0$ be its annihilator and define $\fa^{\phi} := \pr (\g_{-}, \mathfrak{ann}(\phi))$.
 \end{defn}

 In terms of the effective part $\bbE \subset H^2_+(\g_-,\g)$, we define
 \begin{align} \label{def:U}
 \mathfrak{U} &:=\max\left\{\dm \fa^{\phi}:\,0\ne \phi \in \bbE \right\}.
 \end{align}
 Clearly $\mathfrak{U} < \dm \g$.  (Otherwise $\fa^{\phi}=\g$ for some $0 \neq \phi \in \bbE$, and so $\sfZ \in  \mathfrak{ann}(\phi)$.  But necessarily $\sfZ$ acts non-trivially since $\phi \in H^2_{+}(\g_{-},\g)$,  which is a contradiction.)  We will show that $\mathfrak{S} \le \mathfrak{U}$.

 \begin{lem} \label{th:tractorform}
	Let $(\cG \to M,\omega)$ be a regular, normal Cartan geometry of type $(G,P)$.  Let $u \in \cG$ be arbitrary. Let $\xi \in \finf(\cG,\omega)$ with $\omega(\xi_u) \in \g^{1} \subset \mathfrak{p}$  and $\eta \in \Gamma(T^{-1}\cG)^P$. Then:
	\begin{align} \label{fproperty}
	[\omega(\xi_u),\, \omega(\eta_u)] \cdot \kappa_H(u) = 0.
	\end{align}
 \end{lem}

 \begin{proof}
 Fix $u \in \cG$ as above with $A := \omega(\xi_u) \in \g^1$ and $B := \omega(\eta_u) \in \fg^{-1}$.  Since $\xi$ is a symmetry, then $0 = (\mathcal{L}_{\xi} \omega) (\eta) = d\omega(\xi, \eta) + \eta \cdot \omega (\xi) = \xi \cdot \omega(\eta) -\omega([\xi, \,\eta])$.  Evaluation at $u$ now yields 
 \begin{align} \label{E:xieta}
 \omega([\xi, \eta])(u) = (\xi \cdot \omega(\eta))(u) = -[A, B] \in \fp,
 \end{align}
 using $P$-equivariancy of $\omega(\eta)$ and \eqref{E:Pequiv}.

Since $\xi$ is a symmetry, then $\xi \cdot \kappa = 0$ and $\xi \cdot \kappa_H = 0$.  We get the prolonged equation
 \begin{align}\label{cond0}
 0 &= \eta \cdot (\xi \cdot \kappa_H) = \xi \cdot (\eta \cdot \kappa_H)  + [\eta,\,\xi]\cdot \kappa_H.
 \end{align}
 Now evaluate at $u$:
 \begin{itemize}
 \item Since $\eta$ is $P$-invariant and $\kappa_H$ is $P$-equivariant, then $\eta \cdot \kappa_H : \cG \to \frac{\ker \partial^*}{\img \partial^*}$ is $P$-equivariant.  Thus, $(\xi \cdot (\eta \cdot \kappa_H))(u) = -A \cdot (\eta  \cdot \kappa_H)(u) =0$ using \eqref{E:Pequiv} and Lemma \ref{lem:compreducibility} (since $A \in \fg^1$).
 \item Since $[\xi,\eta]$ is $P$-invariant with $\omega([\xi,\eta])(u) \in \fp$, then 
 \begin{align}
 0 \stackrel{\eqref{cond0}}{=} ([\eta,\,\xi]\cdot \kappa_H)(u) \stackrel{\eqref{E:Pequiv}}{=} \omega([\xi,\eta])(u) \cdot \kappa_H(u) \stackrel{\eqref{E:xieta}}{=} -[A,B] \cdot \kappa_H(u).
 \end{align}
 \end{itemize}
 \end{proof}

\begin{thm} \label{th:uub}
	Let $(\pi : \cG \to M, \omega)$ be a regular, normal Cartan geometry of type $(G, P)$ associated to an ODE.  For any $u \in \cG$, we have $\fs(u) \subseteq \fa^{\kappa_H(u)}$.  Moreover, $\mathfrak{S} \le \mathfrak{U} < \dm \g$.
\end{thm}
 \begin{proof}
 Fix any $u \in \cG$.  We have $\fs_0(u) \subseteq \mathfrak{ann}(\kappa_H(u))$ from Proposition \ref{prop:filtration}(iv), so for the first claim it suffices to prove that $\fs_1(u) \subseteq \fa_1^{\kappa_H(u)}$.  Suppose $\fs_1(u) \neq 0$, then we must have $\fs_1(u) = \R \sfY$.  Pick any $B \in \fg_{-1}$.  Let $\xi  \in \finf(\cG,\omega)$ and $\eta \in \Gamma(T^{-1}\cG)^P$ with $\omega(\xi_u) = \sfY$ and $\omega(\eta_u) = B$. Then \eqref{fproperty} with $A := \sfY$ implies that $[\sfY, B] \cdot\kappa_H(u) = 0$, hence $\sfY \in \fa_1^{\kappa_H(u)}$ and the first claim follows. We deduce that $\dm\finf(\cG,\omega) = \dm \fs(u) \le \dm \fa^{\kappa_H(u)} \leq \fU$, since $\kappa_H$ is valued in the effective part $\bbE$.  We conclude that $\mathfrak{S} \le \mathfrak{U} < \dm \g$.
\end{proof}

 \begin{rem}\label{R-BT} In the parabolic setting, the analogous statement $\fs(u) \subseteq \fa^{\kappa_H(u)}$ was proved in  \cite[\S 3]{KT2017} on an open dense set of so-called {\sl regular points} (using a Frobenius integrability argument).  This was strengthened to all points in \cite{KT2018} using the fundamental derivative and calculus on the adjoint tractor bundle.  Our proof in this section is adapted from the latter, but can be formulated and proven more simply since the positive part $\g_+ = \g_1$ consists of only a single grading level (with dimension one).
 \end{rem}

 Let $\cO \subset \bbE$ be a $G_0$-invariant  subset.  We define $\fS_\cO$ analogously to $\fS$ from \eqref{D-submax}, but with the additional constraint that $\kappa_H$ is valued in $\cO$.  We also set
 $\fU_\cO := \max\{ \dim \fa^\phi : 0 \neq \phi \in \cO \}$.  The same argument as in Theorem \ref{th:uub} allows us to conclude:
 \begin{align}
 \fS_\cO \leq \fU_\cO.
 \end{align}
 Of particular interest to us will be the case where $\cO \subset \bbE$ is a $G_0$-irrep $\bbU$, so that $\fS_\bbU \leq \fU_\bbU$.
 
 Suppose that $\bbE = \bigoplus_i \bbU_i$ is the decomposition into $G_0$-irreps $\bbU_i$, which exists since $G_0$ is reductive.  From the definition of $\fU$ and $\fU_{\bbU_i}$, we remark that the following equality is immediate:
 \begin{align} \label{E:U-decomp}
 \fU = \max_i \,\fU_{\bbU_i}.
 \end{align}
 A priori, the corresponding statement $\fS = \max_i \fS_{\bbU_i}$ may not hold, in particular when $\fS_{\bbU_i} \neq \fU_{\bbU_i}$.  Furthermore, submaximally symmetric models may exist with $\kappa_H$ not concentrated along a single irreducible component.
 
%%%%%%%%%%%%%%%%%%%%%%%%%%%%%%%%%%%%%%%%%%%%%%%%%%%%%%%%%
 \section{Computation of upper bounds} 
 \label{sect:main}
%%%%%%%%%%%%%%%%%%%%%%%%%%%%%%%%%%%%%%%%%%%%%%%%%%%%%%%%%

 In this entirely algebraic section, we compute $\fU$ and $\fU_\bbU$ for each $\g_0$-irrep $\bbU \subset \bbE \subset H^2_+(\g_-,\g)$. In view of Theorem \ref{th:uub}, these provide upper bounds on the respective submaximal symmetry dimensions $\fS$ and $\fS_\bbU$.
 
%%%%%%%%%%%%%%%%%%%%%%%%%%%%%%%%%%%%%%%%%%%%%%%%%%%%%%%%%
 \subsection{Bi-gradings} 
 \label{S:BG}
%%%%%%%%%%%%%%%%%%%%%%%%%%%%%%%%%%%%%%%%%%%%%%%%%%%%%%%%% 

 In \eqref{E:grading}, we introduced a $\g_0$-invariant splitting on $\g_{-1}$.  Such splittings similarly arise for parabolic geometries (with respect to non-maximal parabolic subgroups).  Analogously as in that setting \cite{KT2017}, we refine the grading to a {\sl bi-grading}.  Define $\sfZ_1, \sfZ_2 \in \mathfrak{z}(\fg_0)$ with $\sfZ = \sfZ_1 + \sfZ_2$ (see \eqref{E:gr-elt}) by
 \begin{align} \label{bi-grading}
 \sfZ_1 = -\frac{1}{2}(\sfH + n\id_m), \quad \sfZ_2 = -\id_m.
 \end{align}
 We refer to the ordered pair $(\sfZ_1,\sfZ_2)$ as the {\sl bi-grading element}, and then the joint eigenspaces $\g_{a,b} := \{ x \in \g : [\sfZ_1,x] = ax,\, [\sfZ_2,x] = bx \}$ 
 define the bi-grading $\g = \bigoplus_{(a,b) \in \Z^2} \g_{a,b}$.  Note that $\g_0 = \g_{0,0}$ and $\g_{-1} = \g_{-1,0} \oplus \g_{0,-1}$, and we visualize the bi-grading as in Figure \ref{F:bi-grading}.
 
 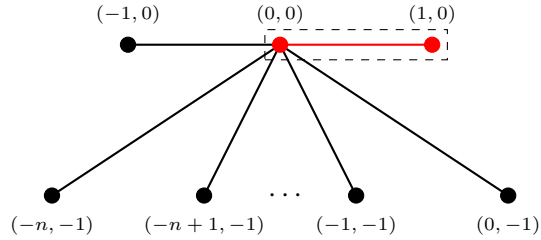
\begin{figure}[h]
 \[
 \begin{tikzpicture}[scale=2,baseline=-3pt]
 \draw (0.05,-1) ++(0,0) node {$\cdots$}; 
 \draw[thick,red] (0,0) -- (1,0);
 \draw[thick] (0,0) -- (-1,0);
 \draw[thick] (0,0) -- (-1.5,-1);
 \draw[thick] (0,0) -- (-0.5,-1);
 \draw[thick] (0,0) -- (0.5,-1);
 \draw[thick] (0,0) -- (1.5,-1);
 \filldraw[red] (1,0) circle (0.05);
 \filldraw[red] (0,0) circle (0.05);
 \filldraw[black] (1.5,-1) circle (0.05);
 \filldraw[black] (0.5,-1) circle (0.05);
 \filldraw[black] (-0.5,-1) circle (0.05);
 \filldraw[black] (-1.5,-1) circle (0.05);
 \filldraw[black] (-1,0) circle (0.05);
 \draw (-1,0) ++(0,0.2) node {\tiny $(-1,0)$};
 \draw (0,0) ++(0,0.2) node {\tiny $(0,0)$};
 \draw (1,0) ++(0,0.2) node {\tiny $(1,0)$};
 \draw (1.5,-1) ++(0,-0.2) node {\tiny $(0,-1)$};
 \draw (0.5,-1) ++(0,-0.2) node {\tiny $(-1,-1)$};
 \draw (-0.5,-1) ++(0,-0.2) node {\tiny $(-n+1,-1)$};
 \draw (-1.5,-1) ++(0,-0.2) node {\tiny $(-n,-1)$};
 \draw [draw=black,dashed] (-0.1,-0.1) rectangle (1.1,0.1);
 \end{tikzpicture}
 \]
 \caption{Bi-grading on $\g$}
 \label{F:bi-grading}
 \end{figure}
 
 The bi-grading on $\g$ induces a bi-grading on cochains and cohomology (since $\partial$ is $\g_0$-equivariant), in particular on the effective part $\bbE \subset H^2_+(\g_-,\g)$.  Given $(a,b) \in \Z^2$, let $\bbE_{a,b} = \{ \phi \in \bbE : \sfZ_1 \cdot \phi = a \phi, \, \sfZ_2 \cdot \phi = b \phi \}$ be the corresponding joint eigenspace.
 
 We note that $\sfZ_2$ acts on $\bigwedge^{2}(\sfrac{\g}{\fp})^{\ast}\tensor \g$ with eigenvalues ($\sfZ_2$-degrees) $0,1$ or $2$.  We will refer to the $G_0$-submodules in $\bbE$ of  {\em positive} $\sfZ_2$-degree as {\sl \Cclass{} modules} and those with {\em zero} $\sfZ_2$-degree as {\sl \Wilc{} modules} (see \S \ref{S:ODE-KH} for this terminology).
 
\begin{defn} \label{D:EC}
Let  $\bbE_{C} \subsetneq \bbE$ denote the direct sum of {\em all}  \Cclass{} modules and $\bbW \subsetneq \bbE$ the direct sum of {\em all}  \Wilc{} modules  in $\bbE$, i.e.\ $\bbE = \bbW \oplus \bbE_{C}$.
\end{defn} 

 \begin{rem} \label{rem:bi-grading}
 In the articles \cite{Doubrov2000, Doubrov2001, Medvedev2010, DM2014} computing the effective part $\bbE$, the gradings on $\g_0$-submodules of $\bbE$ were explicitly stated, but bi-gradings were not used.  However, these can be easily deduced from the cohomology results there (in particular, their realizations as (harmonic) 2-cochains) using the fact that $V$ and $\mathfrak{q}$ have $\sfZ_2$-degrees $-1$ and $0$ respectively.
 \end{rem}
 
%%%%%%%%%%%%%%%%%%%%%%%%%%%%%%%%%%%%%%%%%%%%%%%%%%%%%%%%%
 \subsection{Prolongation-rigidity}
%%%%%%%%%%%%%%%%%%%%%%%%%%%%%%%%%%%%%%%%%%%%%%%%%%%%%%%%%
 
 In view of \S\ref{S:upper}, it is important to understand when the Tanaka prolongation algebra $\fa^\phi$ has non-trivial prolongation in degree +1.

 \begin{lem} \label{L:NPR}
 Let $0 \neq \phi \in \bbE$.  Then $\fa^\phi_1 \neq 0$ if and only if $\phi$ lies in the direct sum of all $\bbE_{a,b}$ for $(a,b)$ that is a multiple of $(n,2)$.
 \end{lem}

\begin{proof} Note that $\fa^\phi_1 \neq 0$ if and only if $\fa^\phi_1 = \g_1 = \mathbb{R} \sfY$.  Since $[\sfY,\g_{0,-1}] = 0$, then this occurs if and only if $[\sfY,\sfX] = -\sfH \in \fa^\phi_0 := \fann(\phi)$.  From \eqref{bi-grading}, we have $\sfH = -2\sfZ_1 + n\sfZ_2$, so $\sfH \in \fann(\phi)$ if and only if $\phi$ lies in the direct sum of the claimed modules.
\end{proof}

 \begin{defn} \label{D:PR}
 We say that a $\g_0$-submodule $\cO \subseteq \bbE$ is {\sl prolongation-rigid (PR)} if $\fa^\phi_1 = 0$ for any $0 \neq \phi \in \cO$.
 \end{defn}

%%%%%%%%%%%%%%%%%%%%%%%%%%%%%%%%%%%%%%%%%%%%%%%%%%%%%%%%%
 \subsection{Scalar case}
 \label{S-EPS}
%%%%%%%%%%%%%%%%%%%%%%%%%%%%%%%%%%%%%%%%%%%%%%%%%%%%%%%%%

 For scalar ODEs, the effective part $\bbE \subset H^2_+(\g_-,\g)$ (Table \ref{tab:scalar}) was computed by Doubrov -- see \cite[Prop.4]{Doubrov2001} for a summary and \cite{Doubrov2000} for details. (Bi-gradings are asserted using Remark \ref{rem:bi-grading}.)  Since $\g_0$ is spanned by $\sfZ_1$ and $\sfZ_2$, then all $\g_0$-irreps $\bbU \subset \bbE$ are 1-dimensional.

 \begin{table}[h] 
 \centering
 \[
 \begin{array}{|cccc|} \hline 
 \mbox{Type} & n & \g_0\mbox{-irrep } \bbU \subset \bbE & \mbox{Bi-grade}  \\ \hline \hline	
 \mbox{\Wilc} & \geq 3 & \underset{(3 \leq r \leq n+1) }{\bbW_r} & (r,0)  \\
\hline		
 \mbox{\Cclass} & 3 & \bbB_3 & (1,2) \\
 & 3 & \bbB_4 & (2,2)  \\
 & 4 & \bbB_6 & (4,2) \\
 & \geq 4 & \bbA_2 & (1,1)  \\
 & \geq 5 & \bbA_3 & (2,1)  \\
 & \geq 6 & \bbA_4 & (3,1) \\ \hline
 \end{array}
 \]
 \caption{Effective part $\bbE \subsetneq H^2_{+}(\g_{-},\g)$ for scalar ODEs of order $n+1 \ge 4$} 
 \label{tab:scalar}
 \end{table}
 
 \begin{lem}\label{lem:uubs} Consider the effective part $\bbE$ for scalar ODEs of order $n+1 \geq 4$.  Then:
 \begin{enumerate}
 \item[(a)] $\bbE$ is not PR if and only if $n = 4$ or $6$.  In particular, $(n,\bbU) = (4,\bbB_6)$ and $(6,\bbA_4)$ are not PR.
 \item[(b)] If $\bbU \subset \bbE$ is a $\g_0$-irrep, then
 $\fU_{\bbU} = \begin{cases}
 n+4, & \mbox{if } (n,\bbU) = (4,\bbB_6) \mbox{ or } (6,\bbA_4);\\
 n+3, & \mbox{otherwise}.
 \end{cases}$
 \item[(c)] $\fU = \begin{cases}
 \fM - 1 = n+4, & \mbox{if } n=4,6;\\
 \fM - 2 = n+3, & \mbox{otherwise}.
 \end{cases}$
 \end{enumerate}
 \end{lem}
 
 \begin{proof} Part (a) directly follows from Lemma \ref{L:NPR} and Table \ref{tab:scalar}.  For part (b), recall that $\dim \g_- = n+2$ and $\dim \fann(\phi) = 1$ for $0 \neq \phi \in \bbU$ since $\bbU$ is irreducible and $\sfZ \not\in \fann(\phi)$ (by regularity).  Thus, $\dim \fa^\phi_{\leq 0} = n+3$, so $\fU_{\bbU} = n+3$ when $\bbU$ is PR and $\fU_{\bbU} = n+4$ when $\bbU$ is not PR (when $(n,\bbU) = (4,\bbB_6)$ or $(6,\bbA_4)$).  Part (c) now follows by using \eqref{E:U-decomp}.
 \end{proof}

\begin{lem} \label{L:CC}
Consider the effective part $ \bbE$ for scalar ODEs \eqref{ODE:sym} of order $n+1 \ge 4$ and $\bbE_{C} = \bigoplus_i\bbU_i \subset \bbE$, the direct sum of all irreducible \Cclass{} modules $\bbU_i$. Then, for $0 \ne \phi \in \bbE_{C}$ such that $\dim \fa^\phi \ge n+3$, we have $\phi \in \bbU_i \subset \bbE_{C}$ for some $i$.
\end{lem}

\begin{proof}
Suppose that for $0 \ne \phi \in \bbE_{C}$, $\dim \fa^\phi \ge n+3$. Since $\dim \g_{-} = \dim \fa^{\phi}_{-} = n+2$, then  $\fa^\phi_0 =\fann (\phi)$ is a non-trivial proper subspace of $\g_{0}$.  Since $\dim \g_{0} = 2$, then $\dim \fa^\phi_0 =1$.  None of the bi-grades for the \Cclass{} modules in Table \ref{tab:scalar} is a multiple of any other, so $\dim \fa^\phi_0 =1$ forces $\phi \in \bbU_i \subset \bbE_{C}$ for some $i$. 
\end{proof}

 %%%%%%%%%%%%%%%%%%%%%%%%%%%%%%%%%%%%%%%%%%%%%%%%%%%
 \subsection{Vector case}
 \label{S-EPV}
 %%%%%%%%%%%%%%%%%%%%%%%%%%%%%%%%%%%%%%%%%%%%%%%%%%%

 For vector ODEs, the effective part $\bbE \subset H^2_+(\g_-,\g)$ (Table \ref{tab:vector}) was computed by Medvedev \cite{Medvedev2011} for the 3rd order case, and Doubrov--Medvedev \cite{DM2014} for the higher order cases.  (Bi-gradings are asserted using Remark \ref{rem:bi-grading}.)  We have $\g_0 = \operatorname{span}\{ \sfZ_1,\sfZ_2 \} \oplus \fsl(W)$, so any $\g_0$-irrep $\bbU \subset \bbE$ is completely determined by its bi-grading and highest weight $\lambda$ with respect to $\fsl(W) \cong \fsl_m$.  The latter can be expressed in terms of the fundamental weights $\lambda_1,\ldots,\lambda_{m-1}$ of $\fsl_m$ with respect to the standard choice of Cartan subalgebra and simple roots.  We note that some of the  modules appearing in \cite{Medvedev2011,DM2014} are not $\g_0$-irreducible, so we have decomposed them here into their trace-free and trace parts.  We also define $\bbW_r := \bbW_r^{\tf} + \bbW_r^{\tr}$ and $\bbA_2 := \bbA_2^{\tf} + \bbA_2^{\tr}$.

 \begin{table}[h] 
 \centering
 \[
 \begin{array}{|cccccc|} \hline 
 \mbox{Type} & n  &\g_0\mbox{-irrep } \bbU & \mbox{Bi-grade}  &  \fsl(W)\mbox{-module } \bbU & \fsl(W) \mbox{ h.w. } \lambda\\ \hline\hline
 \mbox{\Wilc} & \ge 2 & \underset{(2 \leq r \leq n+1)}{\bbW_r^{\tf}} & (r,0)   & \mathfrak{sl}(W) & \lambda_1 + \lambda_{m-1}\\
 & \geq 2 & \underset{(3 \leq r \leq n+1)}{\bbW_r^{\tr}} & (r,0)  &  \mathbb{R} \id_m & 0\\ \hline
 \mbox{\Cclass} & 2 &\bbB_4 & (2,2)  & S^2 W^{\ast} & 2 \lambda_{m-1}\\				
 & \ge 2 &\bbA_2^{\tf} & (1,1)  &  (S^2 W^{\ast}\tensor W)_0 & \lambda_1 + 2\lambda_{m-1}\\
 & \geq 3 &\bbA_2^{\tr} & (1,1)  & W^{\ast} & \lambda_{m-1}\\ \hline
 \end{array}
 \]
 \caption{Effective part $\bbE \subsetneq H^2_{+}(\g_{-},\g)$ for vector ODEs of order $n+1 \geq 3$ with $m \geq 2$}
 \label{tab:vector}
 \end{table}
 
   \begin{table}[h] 
 \centering
 \[
 \begin{array}{|ccc ccc|} \hline 
 \mbox{Type} & n & \g_0\mbox{-irrep } \bbU \subset \bbE & \underset{0 \neq \phi \in \bbU}{\max} \dm\mathfrak{ann}(\phi) &\mbox{Is}\; \bbU \mbox{ PR?} & \mathfrak{U}_{\bbU}\\ \hline\hline	
 \mbox{\Wilc} 
 & \ge 2 & \underset{(2 \leq r \leq n+1)}{\bbW_r^{\tf}} &m^2-2m+3 & \checkmark &  \fM -2m+1 \\
 & \geq 2 & \underset{(3 \leq r \leq n+1)}{\bbW_r^{\tr}} & m^2 & \checkmark &  \fM -2 \\ \hline
 \mbox{\Cclass} 
 & 2 & \bbB_4 & m^2-m+1 & \times &  \fM -m\\			
 & 2 & \bbA_2^{\tf} & m^2-2m+3 & \times &  \fM -2m+2 \\
 & \ge 3 & \bbA_2^{\tf} & m^2-2m +3 & \checkmark &  \fM -2m+1\\
 & \geq 3& \bbA_2^{\tr} & m^2-m+1 & \checkmark &  \fM -m-1\\	\hline
 \end{array}
 \]
  (The contact symmetry dimension of the trivial ODE is $\fM = m^2+(n+1)m+3$.)\\[0.1in]
 \caption{Upper bounds $\fU_\bbU$ for vector ODEs of order $n+1 \geq 3$ with $m \geq 2$}
 \label{tab:lemmaUsystems}
 \end{table}

\begin{lem}\label{+prs} 
 Consider the effective part $\bbE$ for vector ODEs of order $n+1 \geq 3$ with $m \geq 2$.  Then:
 \begin{enumerate}
 \item[(a)] $\bbE$ is not PR if and only if $n=2$.  When $n=2$, $\bbA_2^{\tf}$ and $\bbB_4$ are not PR, while $\bbW_r^{\tf}$ and $\bbW_r^{\tr}$ are PR.
 \item[(b)] If $\bbU \subset \bbE$ is a $\g_0$-irrep, then $\fU_\bbU$ is given in Table \ref{tab:lemmaUsystems}.
 \item[(c)] $\fU = \fM-2 = m^2+(n+1)m+1$.
 \end{enumerate}
\end{lem}

\begin{proof} Part (a) directly follows from Lemma \ref{L:NPR} and Table \ref{tab:vector}.  Let us prove part (b).  In order to compute $\fU_\bbU$, it suffices to maximize $\dim \fann(\phi)$ among $0 \neq \phi \in \bbU$.  (If $\bbU$ is not PR, then $\fa_1^\phi = \R \sfY$ for all $0 \neq \phi \in \bbU$.)  Since $\bbU$ is $\g_0$-irreducible, the maximum is achieved on any highest weight vector $\phi_0$ (and indeed, along the $\SL_m$-orbit through $\phi_0$).  Let $\fu \subset \fsl(W) \cong \fsl_m$ be the parabolic subalgebra preserving $\phi_0$ up to a scaling factor.  Since $\sfZ_1$ and $\sfZ_2$ also preserve $\phi_0$ up to scale, then we obtain
 \begin{align} \label{E:hw-ann}
 \dim \fann(\phi_0) = 1 + \dim \fu.
 \end{align}
 For each $\g_0$-irrep $\bbU \subset \bbE$, the highest $\fsl_m$-weight $\lambda$ and parabolic $\fu \subset \fsl_m$ is given below.
 \begin{align}
 \begin{array}{|c||ccccc|} \hline
 \bbU & \bbW_r^{\tf} & \bbW_r^{\tr} & \bbB_4 & \bbA_2^{\tf} & \bbA_2^{\tr}\\
 \lambda & \lambda_1 + \lambda_{m-1} & 0 & 2\lambda_{m-1} & \lambda_1 + 2\lambda_{m-1} & \lambda_{m-1}\\
 \fu & \fp_{1,m-1} & \fsl_m & \fp_{m-1} & \fp_{1,m-1} & \fp_{m-1}\\ \hline
 \end{array}
 \end{align}
 The subscript notation for parabolics is the same as that used in \cite{KT2017}. (We caution that $\fp$ ornamented with subscripts here is not related to $P$ for the trivial ODE.)  Concretely, each such $\fu$ is a block upper triangular, trace-free $m \times m$ matrix with diagonal blocks of size:
 \begin{itemize}
 \item $1,m-2,1$ for $\fp_{1,m-1}$, so $\dim \fu = m^2-1- 2(m-2)-1 = m^2-2m+2$;
 \item $m-1,1$ for $\fp_{m-1}$, so $\dim \fu = m^2-1- (m-1) = m^2-m$.
 \end{itemize}
 Using $\dim \fg_- = 1 + (n+1)m$ and \eqref{E:hw-ann}, we obtain $\dim \fa^{\phi_0}_{\leq 0}$.  When $\bbU$ is PR, this equals $\fU_{\bbU}$.  When $\bbU$ is not PR, we must augment it by one.  Part (c) now follows by using \eqref{E:U-decomp}.
\end{proof}

%%%%%%%%%%%%%%%%%%%%%%%%%%%%%%%%%%%%%%%%%%%%%%%%%%%%%%%%%
 \section{Submaximal symmetry dimensions}
 \label{S:ODE-KH}
%%%%%%%%%%%%%%%%%%%%%%%%%%%%%%%%%%%%%%%%%%%%%%%%%%%%%%%%%

 For higher order ODEs, we review the known local expressions for $\kappa_H$, labelled here by: 
 \begin{itemize}
 \item $\cW_r$: \quad\,\,\quad {\sl Generalized Wilczynski} invariants (with $\sfZ_2$-degree 0);
 \item $\mathcal{A}_r, \mathcal{B}_r$: \quad {\sl C-class} invariants (with $\sfZ_2$-degrees 1 and 2 respectively).
 \end{itemize}
 These correspond to the $\g_0$-irreps $\bbW_r, \bbA_r, \bbB_r \subset \bbE$ introduced earlier in \S \ref{S-EPS} and \S \ref{S-EPV}.  (The expressions for these invariants were computed with respect to some adapted coframing.  If a different adapted coframing is used, these expressions would transform tensorially according to the structure of the indicated modules.) For each irreducible $\fg_0$-submodule $\bbU \subset \bbE$, we use these differential invariants to exhibit explicit ODE models with abundant symmetries having $\kappa_H$ non-zero and concentrated in $\bbU \subset \bbE$.
 
 For {\em all} vector cases and most scalar cases, these exhibited models realize $\fS_\bbU = \fU_\bbU$, cf.\ Tables \ref{tab:linearmodels}, \ref{tab:submax1} and \ref{tab:submax-Kh}. The contact symmetries of the given ODE models are stated in terms of their projections to $(t,\bu)$-space, i.e.\ $J^0(\R,\R^m)$, in the case of point symmetries, or in terms of their projections to $(t,\bu,\bu_1)$-space, i.e.\ $J^1(\R,\R^m)$, in the case of genuine contact symmetries.  In \S\ref{S:conclusion}, exceptional cases (where $\fS_\bbU < \fU_\bbU$) are discussed and we conclude the proofs of Theorems \ref{T:main} and \ref{T:main2}.

%%%%%%%%%%%%%%%%%%%%%%%%%%%%%%%%%%%%%%%%%%%%%%%%%%%%%%%%%
 \subsection{Generalized Wilczynski invariants}
 \label{S-FundInvL}
%%%%%%%%%%%%%%%%%%%%%%%%%%%%%%%%%%%%%%%%%%%%%%%%%%%%%%%%%

 Consider the class of linear ODEs of order $n+1$:
\begin{align} \label{lODE}
\bu_{n+1} + R_n(t)\bu_{n}+ \ldots + R_1(t)\bu_{1} + R_0(t)\bu = 0,
\end{align}
 where $R_j(t)$ is an $\homo(\R^m)$-valued function.  The invertible transformations
 \begin{equation} \label{Trans}
 (t,\bu) \mapsto (\lambda(t), \mu(t)\bu),\quad\mbox{where}\quad
 \lambda : \R \to \R^\times,\quad \mu : \R \to \operatorname{GL}(m),
 \end{equation} 
 constitute the most general Lie pseudogroup preserving the class \eqref{lODE}.  Using \eqref{Trans}, any equation \eqref{lODE} can be brought into {\sl canonical Laguerre--Forsyth form} defined by $R_n = 0$ and $\tr(R_{n-1}) = 0$.

As proved by Wilczynski \cite{Wilczynski1905} for $m=1$ and Se-ashi \cite{Se-ashi1988} for $m \ge 2$, the following expressions 
\begin{align}\label{I-W}
\Theta_{r} =
\sum_{k=1}^{r-1}(-1)^{k+1}\dfrac{(2r-k-1)!(n-r+k)!}{(r-k)!(k-1)!}R_{n-r+k}^{(k-1)}, \quad r=2,\ldots,n+1,
\end{align}
 are fundamental (relative) invariants with respect to those transformations \eqref{Trans} preserving the Laguerre--Forsyth form.  These invariants are called the {\sl Se-ashi--Wilczynski invariants} and $r$ is the degree of the invariant. We remark that:
 \begin{itemize}
 \item If all $R_j$ are independent of $t$, then all $\Theta_r$ are constant multiples of $R_{n+1-r}$.
 \item For $m=1$ (scalar ODEs), we have $R_{n-1}(t) = 0$ and this forces $\Theta_2 \equiv 0$.
 \end{itemize}
  
The generalized Wilczynski invariants $\cW_r$ directly generalize the Se-ashi--Wilczynski invariants to non-linear ODEs.  We refer to the corresponding modules $\bbW_r$ as being of {\sl Wilczynski-type}.  (Similarly for trace or trace-free parts.)
 
 \begin{defn}\label{D:WI}
 For \eqref{ODE:sym}, $\cW_r$ are defined as $\Theta_{r}$ evaluated at its linearization along a solution $\bu$. Formally, $\cW_r$ are obtained from \eqref{lODE} by substituting $R_r(t)$ by the matrices $\left(-\frac{\partial f^a}{\partial u_r^b}\right)$ and the usual derivative by the total derivative.
 \end{defn}
 
It was proved by Doubrov \cite{Doubrov2008} that $\cW_r$ do not depend on the choice of solution $\bu$ and are indeed (relative) contact invariants of \eqref{ODE:sym}.  Table \ref{tab:linearmodels} exhibits {\em constant coefficient} linear ODEs with $\kappa_H \not \equiv 0$, $\img(\kappa_H) \subset \bbU$ and contact symmetry dimension realizing $\fU_\bbU$, so $\fS_\bbU = \fU_\bbU$ for modules $\bbU$ of Wilczynski type. 

\begin{table} [h]
	\centering
	\[
	\begin{array}{|c|c|c| c|c|c|} 	
	\hline 
	n &m&\bbU &\mbox{ODE with } \img(\kappa_H) \subset \bbU & \mbox{Sym dim} & \mbox{Contact symmetries} \\ 
	\hline	\hline
	\ge 3 & 1 & \underset{(3 \leq r \leq n+1)}{\bbW_r} 
	& u_{n+1} = u_{n+1-r} & \fM-2 
	& \underset{(\{ s_k \}_{k=1}^{n+1} {\tiny \mbox{ solns of } } u_{n+1}=u_{n+1-r})}{\partial_t, \, u \partial_u, \, s_k \partial_u}\\
	\hline
	\geq 2 & \ge 2 & \underset{(3 \leq r \leq n+1)}{\bbW_r^{\tr}} & \underset{(1 \leq a \leq m)}{u_{n+1}^a = u_{n+1-r}^a} & \fM -2 & 
	\underset{(1 \leq a,b \leq m; \,\,\, \{ s_k \}_{k=1}^{n+1} {\tiny \mbox{ solns of } } u_{n+1}=u_{n+1-r})}{\partial_t, \, u^a \partial_{u^b}, \, s_k \partial_{u^a}}\\
	\hline	 
	\ge 2 & \ge 2 & \underset{(2 \leq r \leq n+1)}{\bbW_r^{\tf}} &
	\underset{(1 \leq a \leq m)}{u_{n+1}^a = u_{n+1-r}^2 \delta_1^a}
	&\fM -2m+1 & \begin{array}{c} 
		\underset{(1 \le a, b \le m,\,\, a \ne 2,\, b\ne 1,\,\, 1 \le i \le n)}{
		\partial_t,\,\, 
		\partial_{u^a},\,\,
		t^i{\partial_{u^a}},\,\,
		u^b\partial_{u^a}, }\\
		t \partial_t + r u^1 \partial_ {u^1},\,\,
		u^1 \partial_{u^1} + u^2 \partial_{u^2},\\
		\underset{(n+1 \le k \le n+r)}{\frac{t^k}{k!}\partial_{u^1} + \frac{t^{k-r}}{(k-r)!}\partial_{u^2}},\\
{\small \mbox{for}}\,\, 2 \leq r \leq n\,\, {\small\mbox{in addition:}}\\
\underset{(0 \le \ell \le n-r)}{ t^\ell \partial_{u^2}}
		\end{array}\\	\hline
	\end{array}
	\]
	 (The contact symmetry dimension of the trivial ODE is $\fM = m^2+(n+1)m+3$.)\\[0.1in]
	\caption{Constant coefficient linear ODEs realizing $\fS_\bbU = \fU_\bbU$ for $\bbU$ of Wilczynski type} 
	\label{tab:linearmodels}
\end{table}

%%%%%%%%%%%%%%%%%%%%%%%%%%%%%%%%%%%%%%%%%%%%%%%%%%
 \subsection{C-class invariants}
 \label{S:Cclass}
%%%%%%%%%%%%%%%%%%%%%%%%%%%%%%%%%%%%%%%%%%%%%%%%%%%%%%%%%

As formulated in \cite{CDT2020}, an ODE \eqref{ODE:sym} is of {\sl C-class} if the curvature of the corresponding canonical Cartan geometry satisfies $\kappa(\sfX,\cdot ) = 0$.  This can be characterized at the harmonic level in terms of the generalized Wilczynski invariants $\cW_r$.  Necessity of all $\cW_r \equiv 0$ follows from \cite[Thm.4.1]{CDT2020}, while sufficiency is established in \cite[Thm.4.2]{CDT2020}.  Here, we abuse the terminology and refer to the modules $\bbA_r,\bbB_r$ and corresponding invariants $\cA_r, \cB_r$ as being of {\sl C-class} type (despite the fact that they are defined in general, even for ODEs that are not of C-class).

Below are the C-class invariants of \eqref{ODE:sym}:

\begin{itemize}
\item {\em Scalar case}: The C-class invariants of $u_{n+1} = f(t,u,u_1,...,u_n)$ were computed by Doubrov \cite{Doubrov2001} (see also \cite[Example 6]{DM2014}):
 \begin{align} \label{E:scalar-inv}		
 \begin{split}
 n=3 :\,\,\, \mathcal{B}_3 &=f_{333}, \\
 n=3 :\,\,\, \mathcal{B}_4 &= f_{233}+ \frac{1}{6} (f_{33})^2 + \frac{9}{8} f_3 f_{333} + \frac{3}{4} \frac{d}{dt}f_{333}, \\
 n=4 :\,\,\, \mathcal{B}_6 &= f_{234}-\frac{2}{3}f_{333} - \frac{1}{2}(f_{34})^2\quad\md\quad \langle \mathcal{A}_2, \mathcal{W}_3\rangle,  \\
 n\geq 4 : \,\, \mathcal{A}_2 &=f_{nn},  \\
 n\geq 5 :\,\, \mathcal{A}_3&=f_{n,n-1}+ \frac{n(n-1)}{(n+1)(n-2)}f_n f_{nn} + \frac{n}{n-2}\frac{d}{dt}f_{nn},  \\
 n \geq 6 : \,\, \mathcal{A}_4 &=f_{n-1,n-1} \quad\md\quad \langle \mathcal{A}_2, \mathcal{A}_3, \cW_3\rangle.
 \end{split}
 \end{align}
 Here, $f_i := \frac{\partial f}{\partial u_i}$, see \eqref{E:EV} for $\frac{d}{dt}$, and $\langle \mathcal{I}\rangle$ denotes the differential ideal generated by an invariant $\mathcal{I}$.

\item {\em Vector case}: For $m\ge 2$, the C-class invariants were computed by Medvedev \cite{Medvedev2011} for $n=2$ and by Doubrov--Medvedev \cite{DM2014} for $n \ge 3$.  Letting $\operatorname{tf}$ refer to the trace-free part, we have:
\begin{align}
\begin{split}
n \geq 2: \,\, {(\mathcal{A}_{2})}^a_{bc} &=\tf\left( \frac{\partial^2{f}^a}{\partial{u_n^b}\, \partial{u_n^c}}\right), \\
n = 2: \,\,\, (\mathcal{B}_4)_{bc}&= -\frac{\partial H_c^{-1}}{\partial u_{1}^b}+ \frac{\partial}{\partial u_2^b} \frac{\partial}{\partial u_2^c}{H^t}-\frac{\partial}{\partial u_2^c}\frac{d}{dt}{H^{-1}_b} -\frac{\partial}{\partial u_2^c}\left(\sum_{a=1}^m H_a^{-1}\frac{\partial f^a}{\partial u_2^b}\right) + 2H^{-1}_b H^{-1}_c,
\end{split}
\end{align}
where
\begin{align}
H_b^{-1}= \frac{1}{6(m+1)} \sum_{a=1}^m \frac{\partial^2{f}^a}{\partial{u_2^a}\, \partial{u_2^b}}, \qquad
 H^t = -\frac{1}{4m} \sum_{a=1}^m \left( \frac{\partial f^a}{\partial u_{1}^a}-\frac{d}{dt}\frac{\partial f^a}{\partial u_{2}^a}+ \frac{1}{3} \sum_{c=1}^m \frac{\partial f^a}{\partial u_2^c} \frac{\partial f^c}{\partial u_2^a}\right).
\end{align}
\end{itemize}

 Tables \ref{tab:submax1} and \ref{tab:submax-Kh} respectively exhibit scalar ODEs and vector ODEs with $\kappa_H \not \equiv 0$,  $\img(\kappa_H) \subset \bbU$ and contact symmetry dimension realizing $\fS_\bbU = \fU_\bbU$ for modules $\bbU$ of C-class type. These ODEs are examples of C-class equations since all $\cW_r \equiv 0$. These scalar ODEs are well-known and stated for example in \cite[pp. 205-206]{Olver1995}, but their harmonic curvature classification was not given there. We remark that for the ODE in the first row of Table \ref{tab:submax1}, the $\kappa_H$-classification is deduced from the invariants when $n=3$.  For $n \geq 4$ however, $\img(\kappa_H) \subset \bbA_2$ cannot be asserted by using the invariants alone since $\cB_6$ and $\cA_4$ were computed only up to a differential ideal containing $\cA_2$, and we have $\cA_2 \ne 0$ for this ODE (and $\cA_3 \equiv 0$ for $n \geq 5$).  However, since the ODE admits an $(n+3)$-dimensional contact symmetry algebra, then by Lemma \ref{L:CC} the conclusion $\img (\kappa_H) \subset \bbA_2$ follows.
 
 \begin{table} [h]
	\centering
	\[
	\begin{array}{|c|c| c|c|c|} 	
	\hline 
	n &\bbU &\mbox{ODE with}\, \img(\kappa_H) \subset \bbU & \mbox{Sym dim} & \mbox{Contact symmetries} \\ 
	\hline	\hline
	3 & \bbB_4 & \multirow{3}{*}{$nu_{n-1}u_{n+1}-(n+1)(u_{n})^2 = 0$} & \multirow{3}{*}{$\fM-2 = n+3$} & 
 	\multirow{2}{*}{$\begin{array}{c}
\partial_t,\,\,\,\partial_u,\,\,\,t\partial_t,\,\,\,u\partial_u,\\
	{t^2}\partial_t + (n-2)tu\partial_u,\\ 
	t\partial_u,\,\,\, \ldots,\,\,\, t^{n-2}\partial_u
	\end{array}$}\\[0.5cm]
	\ge 4&\bbA_2&&&\\
	\hline
	4 & \bbB_6 & 9 (u_2)^2 u_5 - 45 u_2 u_3 u_4 + 40(u_3)^3 = 0 & \fM-1 = 8 &  \begin{array}{c}
	\partial_t,\,\,\, \partial_u,\,\,\,
 t \partial_t, \,\,\, u\partial_t,\,\,\,t\partial_u,\\ u\partial_u,\,\,\,
 tu\partial_t + u^2\partial_u,\\
  t^2 \partial_t + (n-3)tu\partial_u
	\end{array}
	\\ \hline
	6 &\bbA_4 & \begin{array}{c}
	10(u_{3})^3u_{7}-70(u_{3})^2u_{4}u_{6}\\-49(u_{3})^2(u_{5})^2
		+ 280u_{3}(u_{4})^2u_{5}\\-175(u_{4})^4 = 0
		\end{array}   & \fM-1 = 10 & \begin{array}{c}
 \partial_t,\,\,\, 
 \partial_u,\,\,\,
 t\partial_t - u_1\partial_{u_1},\\
 t\partial_u + \partial_{u_1},\,\,\,
 {t^2}\partial_u + 2t\partial_{u_1},\\
 u \partial_u + u_1 \partial_{u_1},\,\,\,
 2u_1\partial_t +{u_1^2}\partial_u, \\
 {t^2}\partial_t + 2tu\partial_u +2u\partial_{u_1},\\
 (2tu_1-2u)\partial_t +{t}u_1^2 \partial_u + {u_1^2}\partial_{u_1},\\ 
 	(2t^2u_1 - 4tu)\partial_t +({t^2}u_1^2 - 4u^2)\partial_u\\
	+ (2tu_1^2-4 u u_{1})\partial_{u_1}
		\end{array} \\
	\hline
	\end{array}
	\]
	\caption{Scalar ODEs realizing $\fS_\bbU = \fU_\bbU$ for $\bbU$ of C-class type} 
	\label{tab:submax1}
\end{table}
    \begin{table}[h] 
 	\centering
 	\[
 	\begin{array}{|c|c |c|c|c|} \hline 
 	n & \bbU & \mbox{ODE with } \img(\kappa_H) \subset \bbU &\mbox{Sym dim} &\mbox{Contact (point) symmetries} \\ 
 	\hline\hline
 	2 &\bbB_4  & \multirow{6}{*}{$\begin{array}{c}
 	\underset{(1 \le a \le m)}{u_{n+1}^a= \displaystyle\frac{(n+1)u_n^1 u_n^a}{nu^1_{n-1}}}	\\
 		\end{array}$}   &\multirow{6}{*}{$\begin{cases}
 	\fM-m,& n=2\\
 	\fM-m-1,& n\ge 3
 	\end{cases}$} & \raisebox{-2mm}{\multirow{3}{*}{$\begin{array}{c}
 	\partial_t,\,\, 
	t \partial_t,\,\, 
	u^1\partial_{u^1},\\ 
	\underset{(1 \le a,\,b \le m,\,\, b\ne 1,\,1 \le i,\,j \le n-1,\,\,j\ne n-1)}{\partial_{u^a},\,\,
	u^a \partial_{u^b},\,\,
	tu^1\partial_{u^b},\,\,
	t^j \partial_{u^1},\,\,
	t^i \partial_{u^b}},\\
 t^2 \partial_t + (n-2) t u^1\partial_{u^1}\\
 \quad\quad\quad + (n-1) t \sum_{a=2}^m u^a \partial_{u^a},\\
 	{\small \mbox{for}}\,\, n=2\,\, {\small\mbox{in addition:}}\,\, u^1 \sum_{a=1}^m u^a \partial_{u^a}\\
 		\end{array}$}}\\[2.0cm]
 	\ge 3 &\bbA_2^{\tr}& & &\\
 	\hline			
 	\ge 2  &\bbA_2^{\tf}  & 
 	\begin{array}{c}
 	\underset{(1 \le a \le m)}{u_{n+1}^a= (u_n^2)^2\delta_1^a}	\\
 		\end{array}
 	& \begin{cases}
 	\fM -2m+2, & n=2\\
 	\fM-2m+1, & n\ge 3
 	\end{cases}
  & \begin{array}{c}
 		{\partial_t},\,\,\underset{(1 \le a,\,b,\, c\, \le m,\,\,a\ne 2,\,\, b \ne 1,\,\, 1 \le i,\,j \le n,\, i\ne n)}{\partial_ {u^c},\,\,t^i\partial_ {u^2},\,\,t^j\partial_{u^a},\,\,u^{b}\partial_{u^a},}\\t{\partial_t} - (n-1)u^1{\partial_{u^1}},\,\,
 	   2u^1{\partial_{u^1}} +u^2{\partial_ {u^2}},\\
 		 2tu^2{\partial_{u^1}} + \frac{t^n(n+1)}{n!} \partial_{u^2},\\
 		 {\small \mbox{for}}\,\, n =2\,\, {\small \mbox{in addition:}}\\ 
 3t^2 \partial_t + 2(u^2)^2 \partial_{u^1} + 6 t \sum_{a=1}^m u^a \partial_{u^a}\\
 		\end{array} \\
 	\hline
 	\end{array}
 	\]
	(The contact symmetry dimension of the trivial ODE is $\fM = m^2+(n+1)m+3$.)\\[0.1in]
 	\caption{Vector ODEs realizing $\fS_\bbU = \fU_\bbU$ for $\bbU$ of C-class-type}
 	\label{tab:submax-Kh}
 \end{table}
 
%%%%%%%%%%%%%%%%%%%%%%%%%%%%%%%%%%%%%%%%%%%%%%%%%%%%%%%%%
 \subsection{Exceptional scalar cases and conclusion} 
 \label{S:conclusion}
%%%%%%%%%%%%%%%%%%%%%%%%%%%%%%%%%%%%%%%%%%%%%%%%%%%%%%%%%
 
By Theorem \ref{th:uub}, we have $\fS_{\bbU} \le \fU_{\bbU}$  and $\fS \le \fU$.  The upper bounds were computed in Lemma \ref{lem:uubs} and \ref{+prs}, from which we obtain (using \eqref{E:U-decomp}):
 \begin{align}
 \fU = \begin{cases}
 \fM - 1, & \mbox{if } m=1,\, n \in \{ 4, 6 \};\\
 \fM - 2, & \mbox{otherwise}.
 \end{cases}
 \end{align}
 These are realized by ODEs in Tables \ref{tab:linearmodels} and \ref{tab:submax1}, so $\fS = \fU$ and Theorem \ref{T:main} is proved.

Let us now turn to completing the proof of Theorem \ref{T:main2}.  The equality $\fS_\bbU = \fU_\bbU$ has already been established for all vector cases and most scalar cases.   The following scalar cases remain:
 \begin{align} \label{E:exc}
 (n,\bbU) = (3, \bbB_3), \quad (\geq 5,\bbA_3), \quad (\geq 7,\bbA_4),
 \end{align}
 for which $\fU_\bbU = \fM-2 = n+3$.  (The $(6,\bbA_4)$ case was treated in Table \ref{tab:submax1}.)  Excluding $n=4$ (for which $\fS = 8$) and $n=6$ (for which $\fS = 10$), we already have $\fS = n+3$ for scalar ODEs of order $n+1 \geq 4$.  From \cite[p.206]{Olver1995}, which relies on results of Lie \cite{Lie1924}, all submaximally symmetric ODEs are either linear (but inequivalent to the trivial ODE $u_{n+1} = 0$) or equivalent to either:
 \begin{align} \label{E:OlvODE}
 n u_{n-1} u_{n+1} - (n+1) (u_n)^2 = 0, \quad\mbox{or}\quad 3 u_2 u_4 - 5 (u_3)^2 = 0.
 \end{align}
 We exclude the linear cases, for which all C-class invariants vanish.  The first ODE in \eqref{E:OlvODE} has already appeared in Table \ref{tab:submax1} (associated to $(3,\bbB_4)$ or $(\geq 4, \bbA_2)$).  The second ODE in \eqref{E:OlvODE} has $\kappa_H$ concentrated in $\bbB_4$ (using the known relative invariants in \S\ref{S:Cclass}). We conclude that 
 \begin{align} \label{E:scU}
 \fS_\bbU \leq n+2 < \fU_{\bbU} =n+3
 \end{align}
   for all cases in \eqref{E:exc} except possibly the $(6,\bbA_3)$ case. The latter case is resolved in \S \ref{S:NAM} (Theorem \ref{T-Exc}) and indeed \eqref{E:scU} also holds in this case.

Let us now exhibit model ODEs with $\kappa_H \not \equiv 0$, $\img(\kappa_H) \subset \bbU $ and all  $\cW_r \equiv 0$ (ODEs of C-class type). The assertions $\cW_r \equiv 0$ and $\img(\kappa_H) \subset \bbU$ are established using Definition \ref{D:WI} and the differential invariants from \S \ref{S:ODE-KH}.\\
\begin{itemize}
\item $(3,\bbB_3)$: The ODE $u_4 = (u_3)^k$ for $k \neq 0,1$ has the 5-dimensional contact symmetry algebra:
 \begin{align}
 \partial_t,\quad \partial_u, \quad t \partial_u, \quad t^2 \partial_u, \quad
 (k-1)t\partial_t + (3k-4)u\partial_u.
 \end{align}
 Generally, both $\cB_3$ and $\cB_4$ are nonzero.  Requiring $\cB_4 = 0$, i.e  $\img (\kappa_H) \subset \bbB_3$ forces $k= \frac{74 + 2 \sqrt{46}}{49}$.  Thus, $\fS_{\bbB_3} = 5 < \fU_{\bbB_3} = 6$.\\

\item $(\geq 5, \bbA_3)$: Consider the following ODE (obtained as $S_{n+1} = 0$ from \cite[p. 475]{Olver1995}):
	\begin{align}\label{M: A3}
	(n-1)^2(u_{n-2})^2u_{n+1}
	-3(n-1)(n+1)u_{n-2}u_{n-1}u_{n}
	+2n(n+1)(u_{n-1})^3 = 0,
	\end{align}
	which has the following $n+2$ contact symmetries when $n \geq 5$:
	\begin{align} \label{E:Sn}
	\begin{array}{c}
	\partial_t,\quad \partial_u, \quad t \partial_t, \quad u\partial_u,\quad  t\partial_u,\quad \ldots,\quad t^{n-3}\partial_u, \quad t^2 \partial_t + (n-3)tu\partial_u.
	\end{array} 
	\end{align}
 (Sidenote: when $n=4$ the ODE \eqref{M: A3} recovers the submaximally symmetric model from Table \ref{tab:submax1} in the $(4,\bbB_6)$ case, which admits eight symmetries: those in \eqref{E:Sn} and additionally $u\partial_t$ and $tu \partial_t + u^2\partial_u$.)

	We have $\cA_2 \equiv 0$ (and $\cW_r \equiv 0$), but $\cA_3 \ne 0$.  When $n=5$, the invariant $\cA_4$ does not arise, so in this case we can assert that $\img(\kappa_H) \subset \bbA_3$ and $\fS_{\bbA_3} = 7 < \fU_{\bbA_3}=8$ (using \eqref{E:scU}).  For $n \geq 6$, since $\cA_4$ was computed only up to the differential ideal $\langle \cA_2, \cA_3, \cW_3 \rangle$, then the formula given in \eqref{E:scalar-inv} for $\cA_4$ is ambiguous, and so we cannot directly use it on \eqref{M: A3}.  From \eqref{E:scU}, we can only assert $\fS_{\bbA_3} \leq n+2 < \fU_{\bbA_3} = n+3$ for $n \geq 6$.  \\
 
\item $(\geq 7, \bbA_4)$:  The ODE $u_{n+1} = (u_{n-1})^2$ admits the following $n+1$ contact symmetries:
\begin{align}
\partial_t, \quad \partial_u, \quad t\partial_u, \quad \ldots,\quad t^{n-2}\partial_u, \quad t\partial_t + (n-3) u\partial_u.
\end{align}
 We confirm that it has vanishing $\cA_2, \cA_3, \cW_3$, so the formula for $\cA_4$ is unambiguous and $\cA_4 \neq 0$, i.e.\ $\kappa_H \not \equiv 0$ and $\img(\kappa_H) \subset \bbA_4$.  Hence, $n+1 \le \fS_{\bbA_4} \le n+2 < \fU_{\bbA_4}=n+3$.\\
 \end{itemize}
 
 This completes the proof of Theorem \ref{T:main2}.  We remark that for $(n, \bbU)=(\ge 6, \bbA_3)$ or $(\ge 7, \bbA_4)$, we currently do not know of any ODE \eqref{ODE:sym} of order $n+1$ with $\kappa_H \not \equiv 0$ and $ \img(\kappa_H) \subset \bbU$ that has contact symmetry dimension $n+2$. (See Remark \ref{R:N7} for further discussion.)  Determining $\fS_{\bbU}$ for these cases remains open.\\
 
 \appendix

%%%%%%%%%%%%%%%%%%%%%%%%%%%%%%%%%%%%%%%%%%%%%%%%%%%%%%%%% 
 \section{Exceptional scalar cases} 
 \label{S-Exc}
%%%%%%%%%%%%%%%%%%%%%%%%%%%%%%%%%%%%%%%%%%%%%%%%%%%%%%%%%

 Fix $(G,P)$ as in \S \ref{S:TODE}, and the effective part $\bbE \subset H^2_+(\fg_-,\fg)$ as given in \S\ref{S-EPS} and \S\ref{S-EPV}.  Let $\bbU \subset \bbE$ be a $\g_{0}$-irrep. Recall from \S \ref{S:ODE-KH} that for ODEs \eqref{ODE:sym} of order $n+1$  with $\kappa_H \not \equiv 0$  and $\img(\kappa_H) \subset \bbU $, an algebraic upper bound $\fU_{\bbU}$ on the submaximal symmetry dimension $\fS_{\bbU}$ is realizable for all vector cases and the majority of scalar cases.  Among the remaining scalar cases $(n,\bbU) = (3, \bbB_3), (\ge5,\bbA_3)$ or $(\ge7, \bbA_4)$, we asserted that $\fS_{\bbU} < \fU_{\bbU}$ for all of these in \S \ref{S:conclusion}, except for $(6,\bbA_3)$, based on the known classification of submaximally symmetric scalar ODEs as described in \cite[p. 206]{Olver1995}.   In this section, we outline a Cartan-geometric method for establishing $\fS_\bbU < \fU_\bbU$ for the exceptional scalar cases, and in particular establish $\fS_{\bbA_3} < \fU_{\bbA_3}$ for $n=6$  (Theorem \ref{T-Exc}).

%%%%%%%%%%%%%%%%%%%%%%%%%%%%%%%%%%%%%%%%%%%%%%%%%%%%%%%%%
 \subsection{Local homogeneity and algebraic models} 
 \label{S: AM}
%%%%%%%%%%%%%%%%%%%%%%%%%%%%%%%%%%%%%%%%%%%%%%%%%%%%%%%%%

 \begin{lem}\label{L:LH}  
 	For regular, normal Cartan geometries of type $(G, P)$ and a $\g_{0}$-irrep $\bbU \subset \bbE$, suppose that $\fS_{\bbU} = \fU_{\bbU}$. Then any geometry $( \cG \to M, \omega)$ with $\kappa_H$ valued in $\bbU$ and $\dim \finf(\cG, \omega) = \fU_{\bbU}$ is locally homogeneous near any  $u \in \cG$ with $\kappa_H(u)\ne 0$.
 \end{lem}
 \begin{proof}  Fix $u \in \cG$.  By Theorem \ref{th:uub}, $\fs (u) \subset \fa^{\kappa_H(u)}$.  Then by definition of $\fU_\bbU$,
 	\begin{align}
 		\fS_{\bbU} &:=\dm \finf(\cG,\omega) = \dm \fs(u) 
 		\le \dm \fa^{\kappa_H(u)} \le \fU_{\bbU}.
 	\end{align}
 	So, $\fS_{\bbU} = \fU_{\bbU}$ implies $\fs(u) = \fa^{\kappa_H(u)} \supset \g_{-}$. The result then follows by Lie's third theorem.
 \end{proof}
 
 It is well known that a homogeneous Cartan geometry $(\pi : \cG \to M, \omega)$ of fixed type $(G,P)$ can be encoded by algebraic data \cite[Prop 1.5.15]{Andreas-Jan2009}. Fix $u \in \cG$ and let $F^0 \subset F$ denotes  stabilizer of a point $\pi(u) \in M$ and let $\ff^0$ and  $\ff$ denote  the Lie algebras of $F^0$ and $F$ respectively. Then the induced $F$-action on $M$ is  transitive. Any $F$-invariant Cartan connection is completely determined by some distinguished linear map  $\varpi : \ff \to \g $ ({\sl an algebraic Cartan connection} of type $(\g, P)$).  In particular, $\varpi|_{\ff^0}$ is a Lie algebra homomorphism, so $\ker (\varpi)\subset \ff^0$ is an ideal in $\ff$. Since the action of $F$ on $F/F^0$ can be assumed to be {\sl infinitesimally effective} (i.e.\ $\ff^0$  does not contain any non-trivial ideals of $\ff$), then without loss of generality we can restrict to injective maps $\varpi$. Consequently, we can identify $\ff$ with its image $\varpi(\ff)$ in $\g$. Analogous  to \cite[Defn 2.5]{The2021} and in light of the fact that canonical Cartan connections for ODEs satisfy the strong regularity condition (Remark \ref{R-G0}), any homogeneous Cartan geometry arising from an ODE can be encoded as:
 \begin{defn}\label{D:AM}
 	An \emph{algebraic model} $(\ff;\g,\fp )$ of ODE type is a Lie algebra $(\ff, [\cdot,\cdot]_{\ff})$ satisfying: 
 	\begin{itemize}
 		\item[(A1)]$\ff \subset \g$ is a filtered linear subspace such that $\ff^i = \g^i \cap \ff $ and $\fs := \gr (\ff)$ with $\fs_{-}=\g_{-}$;
 		\item[(A2)] $\ff^0$ inserts trivially into ${\kappa} (X, Y) := [X, Y] - [X, Y]_{\ff}$, i.e.  ${\kappa}(Z, \cdot)= 0 \quad\forall Z \in \ff^0$
 		\item[(A3)] $\partial^{\ast} \kappa = 0$  and $\kappa(\g^i, \g^j) \subset \g^{i+j+1} \cap \g^{\min(i,j)-1} \quad \forall i,j$. 
 	\end{itemize}
 \end{defn}
 
 Recall from \S  \ref{subsec:Cartangeom} that  $\kappa_H := \kappa \mod \img \partial^{\ast}$, where $\partial^{\ast}$ is the adjoint of the Lie algebra cohomology differential with respect to a natural inner product on $\g$.
 
 \begin{prop} \label{P:NC}
 	Let $(\ff;\g,\fp )$ be an algebraic model of ODE type. Then
 	\begin{itemize}
 		\item[(a)] $(\ff, [\cdot, \cdot]_{\ff})$ is a filtered Lie algebra. \label{f:LA}
 		\item[(b)] $\ff^0 \cdot \kappa = 0$, i.e.\  $[Z, \kappa(X, Y)]_{\ff} = \kappa([Z, X]_{\ff}, Y) + \kappa(X, [Z, Y]_{\ff})$,\,\,$\forall X,Y \in \ff$ and $\forall Z \in \ff^0$. \label{k:ann}
 		\item[(c)] $\fs \subset \fa^{\kappa_H}$. \label{s:cd}
 	\end{itemize}
 \end{prop}
 
 \begin{proof}
 	This is the same as for corresponding statements in the parabolic geometry setting \cite[Prop 2.6]{The2021}.
 \end{proof}
 
 Fix $(G, P)$ and denote by $\cN$ the set of all algebraic models   $(\ff;\g,\fp )$ of ODE type. Then $\cN$:
 \begin{enumerate}
 	\item admits $P$-action: for $p \in P$  and $ \ff \in \cN$, $p\cdot \ff := \Ad_{p}(\ff)$. All algebraic models belonging to the same $P$-orbit are considered to be equivalent.
 	\item  a partially ordered set with relation $\le$ defined as follows: for $\ff, \widetilde{\ff} \in \cN$ regard $\ff \le \widetilde{\ff}$ if there exists an injection $\ff \hookrightarrow \widetilde{\ff}$ of Lie algebras. We will focus on maximal elements $\ff$ (for this partial order).
 \end{enumerate}
 \begin{rem}
 	By \cite[Lemma 4.1.4]{KT2017}, to each algebraic model $(\ff; \g, \fp)$ of ODE type, there exists a locally homogeneous  geometry $(\cG \to \cE, \omega)$ of type $(G,P)$ with $\finf(\cG, \omega)$ containing a subalgebra isomorphic to $\ff$. Moreover, if $\ff$ is maximal, then it is isomorphic to $\finf(\cG, \omega)$.
 \end{rem}
 
 At the level of vector spaces, $\ff \subset \g$ can be  understood as the graph of a linear map  on $\fs$ into some subspace $\fs^{\perp}  \subset \fp$ with $\g = \fs \oplus \fs^{\perp}$ as follows. Choosing such a graded subspace $\fs^{\perp}$, we can write
 \begin{align}
 \ff := \bigoplus_i \left\langle x + \fD(x) : x \in \fs_i\right\rangle,
 \end{align}
 for some {\it unique} linear ({\sl deformation}) map $\fD : \fs \to \fs ^{\perp}$ satisfying $\fD(x)\in \fs^{\perp} \cap \g^{i+1}$ for $x \in \fs_i$.  For $\widehat{x}:= x + \fD(x) \in \ff$, we will refer to $x \in \fs$ as the {\sl leading part} and $\fD(x)$ as the {\sl tail}.
 
 \begin{lem} \label{L:fD}
 	Let $T \in \ff^0$ and suppose that $\fs$ and $\fs^{\perp}$ are $\ad_T$-invariant subspaces. Then $T \cdot \fD = 0$, i.e $\ad_T \circ \fD = \fD \circ \ad_T$.	
 \end{lem}
 
 \begin{proof}
 	Recall $\fs, \fs^{\perp} \subset \g$ are graded. Given $x \in \fs_i$, we have $x +\fD(x) \in \ff $ and $[T, x+\fD(x)]_{\ff} \in \ff$. Since $T \in \ff^0$, then $\kappa(T, \cdot) = 0$ and therefore $[T, x+\fD(x)]_{\ff}= [T, x+\fD(x)] = [T, x] + [T, \fD(x)]$. By $\ad_T$-invariancy of $\fs$ and $\fs^{\perp}$, we have $[T, x]\in \fs$ and $[T, \fD(x)]\in \fs^{\perp} \cap \g^{i+1}$. The uniqueness of $\fD$ then implies $[T, \fD(x)] = \fD([T, x])$.
 \end{proof}

%%%%%%%%%%%%%%%%%%%%%%%%%%%%%%%%%%%%%%%%%%%%%%%%%%%%%%%%% 
 \subsection{Realizability of a curvature-constrained upper bound} 
 \label{S:RU} 
%%%%%%%%%%%%%%%%%%%%%%%%%%%%%%%%%%%%%%%%%%%%%%%%%%%%%%%%%

 From \S \ref{S: AM}, $\fS_{\bbU} = \fU_{\bbU}$ implies local homogeneity (Lemma \ref{L:LH}) and then the problem of realizability of $\fU_{\bbU}$ reduces to that of existence of an algebraic model $(\ff; \g, \fp)$ of ODE type with $\kappa_H \not \equiv 0$, $\img(\kappa_H) \subset \bbU$ and $\dim \ff= \fU_{\bbU}$. 
Recall from \S \ref{S:TODE}, \S \ref{S:BG} and \S \ref{S-EPS}:
\begin{itemize}
\item  basis for $\g \cong (\fsl_2 \times \gl_1) \ltimes (\V_n \tensor \R) $: $\sfX, \sfH, \sfY$ (standard $\fsl_2$-triple), $E_0,\ldots,E_n$ (for $\fsl_2$-irrep module $\V_n$) and $\id_1$. And $\sfZ_1, \sfZ_2$ are the bi-grading elements.
\item $\bbU \subset \bbE$ is one-dimensional with bi-grade $(a,b) = (1,2), (2,1), (3,1)$ for $\bbB_3, \bbA_3, \bbA_4$ respectively.  Thus, for any $0 \neq \phi \in \bbU$, we have $\fann(\phi) = \langle T:= b\sfZ_1-a\sfZ_2\rangle$. Since $\bbU$ is prolongation rigid (Lemma \ref{lem:uubs}), then $\fa^\phi_1 = 0$ for any $0 \neq \phi \in \bbU$. So,  $\fa :=\fa^{\phi} = \g_{-} \oplus \fann(\phi) \subset \g$, is a  graded subalgebra of dimension $n+3$.
\end{itemize}

 \begin{prop} \label{P:DR}
 	Fix $(n,\bbU) = (3, \bbB_3), (\geq 5,\bbA_3)$ or $(\geq 7, \bbA_4)$. If there exists an algebraic model $(\ff;\fg,\fp)$ of ODE type with  $\kappa_H \not \equiv 0$, $\img(\kappa_H) \subset \bbU$ and $\dim \ff = \fU_\bbU = n+3= \fM-2$,  then fixing $0 \neq \phi \in \bbU$ and using the $P$-action $\ff \mapsto \Ad_p \ff$, we may normalize to $\ff = \fa^\phi$ as filtered vector spaces.
 \end{prop}
 
 \begin{proof}
 	Suppose such an algebraic model with $\fs := \gr(\ff)= \fa^\phi$ exists. Let $\widehat{T} \in \ff^0$ with leading part $T$, so $\widehat{T} := b\sfZ_1 -a\sfZ_2 + \lambda \sfY$.  We use the $P_{+}$- action to normalize $\lambda = 0$:
 	\begin{align}
 		\Ad_{\exp(t\sfY)}(\widehat{T}) &= \exp(\ad_{t\sfY})(\widehat{T})= \widehat{T} + [t\sfY,\widehat{T}]+ \frac{1}{2!}[t\sfY,[t\sfY, \widehat{T}]] + \cdots
 		= b\sfZ_1 - a\sfZ_2 + (\lambda -bt)\sfY.
 	\end{align}
For our cases of interest, $(a,b) \in \{ (1,2), (2,1), (3,1) \}$, so $b\neq 0$ and choosing $t = \frac{\lambda}{b}$ normalizes $\widehat{T} = T$. So, $T = b\sfZ_1 -a\sfZ_2 \in \ff^0$ and by property (A2) of Definition \ref{D:AM},  we have
 	$[T, \cdot]_{\ff} := [T, \cdot]$. Consequently, $\fs$ and $\fs^{\perp} := \langle \sfZ_1, \sfY \rangle$ are $\ad_T$-invariant graded subspaces of $\g$, so by Lemma \ref{L:fD}, the deformation map $\fD : \fs \to \fs^{\perp}$ satisfies $T \cdot \fD = 0$.
	
We claim that $\fD = 0$.  Equivalently, for $\widehat{\sfX}, \widehat{E}_i \in \ff$ with leading parts $\sfX, E_i$ respectively, we claim that $\widehat\sfX = \sfX$ and $\widehat{E}_i = E_i$.  First focus on $\widehat\sfX$ and $\widehat{E}_n$, whose tails are valued in $\fs^\perp = \langle \sfZ_1, \sfY \rangle$.  Recall from Figure \ref{F:bi-grading} that $\sfX, E_n, \sfZ_1, \sfY$ are of bi-grades $(-1,0),(0,-1),(0,0),(1,0)$.  Letting $\omega^n$ and $\omega^{\sfX}$ denote dual basis elements to $E_n$ and $\sfX$ respectively, the eigenvalues of $T = b\sfZ_1 - a\sfZ_2$ acting on 
 \begin{align}
 \omega^{n} \tensor \sfZ_1, \quad \omega^{n} \tensor \sfY, \quad \omega^{\sfX} \tensor \sfZ_1, \quad \omega^{\sfX} \tensor\sfY
 \end{align}
 are $-a,-a+b, b, 2b$.  None of these are zero, so the condition $T \cdot \fD = 0$ forces $\fD(\sfX) = 0 = \fD(E_n)$ and hence $\widehat\sfX = \sfX$ and $\widehat{E}_n = E_n$.  Any ODE with $\kappa_H$ concentrated in any of the  C-class modules $\bbB_3,\bbA_3,\bbA_4$ is of C-class, so $\kappa({\sfX}, \cdot) = 0$ (see discussion in \S\ref{S:Cclass}) and $[{\sfX}, \cdot]_{\ff} = [{\sfX}, \cdot]$.  Since $\sfX, E_n \in \ff$, then $\ff \ni [\sfX,E_i]_\ff = [\sfX,E_i] = E_{i-1}$ inductively from $i=n$ to $i=1$.  Thus, $\widehat{E}_i = E_i \; \forall i$, so $\fD = 0$ and $\ff = \fs = \fa^\phi$.
 \end{proof}

%%%%%%%%%%%%%%%%%%%%%%%%%%%%%%%%%%%%%%%%%%%%%%%%%%%%%%%%% 
 \subsubsection{Non-existence of algebraic models for the exceptional scalar cases}
 \label{S:NAM}
%%%%%%%%%%%%%%%%%%%%%%%%%%%%%%%%%%%%%%%%%%%%%%%%%%%%%%%%%

We prove that for $(n, \bbU) = (6, \bbA_3)$, there are no algebraic models $(\ff; \fp, \g)$ with $\kappa_H \not \equiv 0$, $\img(\kappa_H) \subset \bbA_3$, and $\dim \ff = \fU_{\bbA_3}$. Thus, $\fU_{\bbA_3}$ is not realizable, i.e.\ $\fS_{\bbA_3} < \fU_{\bbA_3}$ (Theorem \ref{T-Exc}).  From \S  \ref{subsec:Cartangeom}, $\kappa_H := \kappa \mod \img \partial^{\ast}$ with $\partial^* \kappa = 0$, but the determination of $\partial^*$ is rather tedious, requiring specific information about the inner product on $\g$.  We have not provided details of this in our article since for our purposes here they can be completely circumvented.  Namely in the proof of Theorem \ref{T-Exc}, instead of showing that ``normal filtered deformations'' provided by $\kappa$ do not exist, we show that arbitrary ``filtered deformations'' do not exist.  In a similar manner, $\fS_{\bbU} < \fU_{\bbU}$ can be established for $(n,\bbU) = (3, \bbB_3),  (5, \bbA_3),  (\geq 7, \bbA_3)$ or $(\geq 7, \bbA_4)$.

  \begin{thm}\label{T-Exc}
  There are no  algebraic models $(\ff;\fg,\fp)$ for seventh order ODEs \eqref{ODE:sym} with $\kappa_H \not \equiv 0$, $\img(\kappa_H) \subset \bbA_3$ and $\dim \ff = \fU_{\bbA_3} = 9$. Thus, $\fS_{\bbA_3} \le 8$. 
 \end{thm}
 \begin{proof} Note that $n=6$. Fix $0 \neq \phi \in \bbA_3$ (bi-grade $(2,1)$), $\fa := \fa^\phi$, and $\fa_0 = \langle T \rangle$, where $T = \sfZ_1 - 2\sfZ_2$.  Assume there is an algebraic model $(\ff;\fg,\fp)$ of ODE type with $\fs := \gr(\ff)= \fa$. By Proposition \ref{P:DR}, we may assume that $\ff = \fa$. Let $\{ \omega^0, \ldots, \omega^n,\omega^\sfX \}$ denote the dual basis to $\{ E_0,\ldots, E_n, \sfX \}$. We note that any $\beta \in \wedge^{2}(\sfrac{\g}{\fp})^{\ast}\tensor \g$ has $\sfZ_2$-degree at most 2. Since $\ff^0 \cdot \kappa = 0$ (Proposition \ref{P:NC} (b)) and $\kappa(\sfX, \cdot) = 0$ (the ODE is of C-class), then $\kappa$ is a linear combination of the 2-cochains below:
 	\[
 	\begin{array}{|c|c|}
 	\hline 
 	\mbox{Bi-grade}  & \mbox{2-cochains}\\
 	\hline \hline
 	(2,1) & \begin{array}{l}
 	\omega^0 \wedge \omega^4 \tensor E_0, \quad
	\omega^1 \wedge \omega^3 \tensor E_0, \quad 
	\omega^0 \wedge \omega^5 \tensor E_1, \quad 
	\omega^1 \wedge \omega^4 \tensor E_1, \\
	\omega^2 \wedge \omega^3 \tensor E_1, \quad
 	\omega^0 \wedge \omega^6 \tensor E_2, \quad
	\omega^1 \wedge \omega^5 \tensor E_2, \quad
	\omega^2 \wedge \omega^4 \tensor E_2, \\
	\omega^1 \wedge \omega^6 \tensor E_3, \quad
	\omega^2 \wedge \omega^5 \tensor E_3, \quad
 	\omega^3 \wedge \omega^4 \tensor E_3, \quad
	\omega^2 \wedge \omega^6 \tensor E_4, \\
	\omega^3 \wedge \omega^5 \tensor E_4, \quad
	\omega^3 \wedge \omega^6 \tensor E_5, \quad
	\omega^4 \wedge \omega^5 \tensor E_5, \quad
 	\omega^4 \wedge \omega^6 \tensor E_6
 	\end{array} \\ \hline
 	(4,2)& \begin{array}{l}
 	 \omega^1 \wedge \omega^6 \tensor \sfX,\quad \omega^2 \wedge \omega^5 \tensor \sfX,\quad \omega^3 \wedge \omega^4 \tensor \sfX,\quad \omega^2 \wedge \omega^6 \tensor T,\quad \omega^3 \wedge \omega^5 \tensor T
 	\end{array}\\
 	\hline	
 	\end{array}
 	\]
 We observe that all such 2-cochains are regular and satisfy the strong regularity condition, i.e.\ $\kappa(\g^i, \g^j) \subset \g^{i+j+1} \cap \g^{\min(i,j)-1} \quad \forall i,j$.

 Next, we show that the Jacobi identity for $(\ff, [\cdot, \cdot]_{\ff})$ forces $\kappa \equiv 0$. For all $x,y,z \in \ff$ , define
 \begin{align}
 \Jac^\ff(x,y,z) := [x,[y,z]_{\ff}]_{\ff} -[[x,y]_{\ff}, z]_{\ff} -[y, [x, z]_{\ff}]_{\ff}.
 \end{align}
 For any $y, z \in \ff$, a direct computation shows that $ 0 = \Jac^\ff(\sfX,y,z)= (\sfX \cdot \kappa)(y, z)$.  Expanding this gives many conditions (see the Maple file accompanying the arXiv submission of this article) and this leads to:
\begin{align}
\begin{split}
\kappa &= \lambda \left[(\omega^0 \wedge \omega^4-\omega^1 \wedge \omega^3) \tensor E_0 + (\omega^0 \wedge \omega^5-\omega^2 \wedge \omega^3) \tensor E_1 \right.\\
&\quad\quad+  (\omega^0 \wedge \omega^6+\omega^1 \wedge \omega^5-\omega^2 \wedge \omega^4)\tensor E_2 + (2\omega^1 \wedge \omega^6-\omega^3 \wedge \omega^4)\tensor E_3 \\
&\quad\quad + \left.(2\omega^2 \wedge \omega^6 - \omega^3 \wedge \omega^5)\tensor E_4 + (\omega^3 \wedge \omega^6 - \omega^4 \wedge \omega^5)\tensor E_5\right]\\
&\quad + \mu (\omega^1 \wedge \omega^6 - \omega^2 \wedge \omega^5+\omega^3 \wedge \omega^4)\tensor \sfX.
\end{split}
\end{align}
Then $\Jac^{\ff}(E_2,E_4,E_6) =0$ implies $\lambda = 0$, while $\Jac^{\ff}(E_1,E_2,E_5) =0$ then forces $\mu = 0$, and hence $\kappa \equiv 0$. Thus, an algebraic model with $0 \neq \kappa_H \subset \bbA_3$ with $\dim \ff = \fU_{\bbA_3}$ does not exist.
 \end{proof}

 \begin{rem}\label{R:N7}
 Fix  $(n, \bbU)=(\ge 6, \bbA_3)$ or $(\ge 7, \bbA_4)$ and recall from Table \ref{tab:scalar} that  the bi-grades $(a, b) $ for the \Cclass{} modules $\bbA_3$ and $\bbA_4$ are $(2,1)$ and $(3,1)$, respectively. Then  from \S \ref{S:conclusion}, we have  
 \begin{align} \label{E:A4}
\fS_{\bbU} \le  n+2 < \fU_{\bbU} = n+3 = \fM-2.
 \end{align} 
 For $0 \ne \phi \in \bbU$, we have $\fa:=\fa^{\phi} = \g_{-} \oplus \fann(\phi) =\g_{-} \oplus \langle T:= b\sfZ_1 -a\sfZ_2 \rangle \subset \g$, which is a graded subalgebra of dimension $n+3$.  If there exists an ODE whose associated Cartan geometry $(\cG \to M, \omega)$ satisfies $0 \ne \kappa_H(u) \in \bbU$,\, $\forall u \in \cG$, then  from Theorem \ref{th:uub} we have the graded Lie algebra inclusion 
 \begin{align}
 \fs(u) \subset \fa^{\kappa_H(u)} = \fa, \quad \forall u \in \cG.
 \end{align}
 By \eqref{E:A4}, this inclusion is proper and a priori we do not need to have $\fg_- \subset \fs(u)$.  If the contact symmetry dimension is $\fU_{\bbU}-1 = n+2$, then there are three possibilities to investigate:
 \begin{enumerate}[(i)]
 \item inhomogeneous case: $\fs(u) = \langle E_0, \ldots, E_n, T \rangle$;
 \item inhomogeneous case: $\fs(u) = \langle E_0, \ldots, E_{n-1}, X, T \rangle$;
 \item homogeneous case: $\fs(u) = \langle E_0, \ldots, E_n, X \rangle = \fg_-$.
 \end{enumerate}
 Thus, identifying $\fS_{\bbU}$ is more difficult and at this point we can only assert that:
 \begin{align} 
 \fS_{\bbA_3} \le n+2, \qquad n+1 \le \fS_{\bbA_4} \le n+2.
 \end{align}
\end{rem}

%%%%%%%%%%%%%%%%%%%%%%%%%%%%%%%%%%%%%%%%%%%%%%%%%%%%%%%%%
 \section*{Acknowledgements}
%%%%%%%%%%%%%%%%%%%%%%%%%%%%%%%%%%%%%%%%%%%%%%%%%%%%%%%%%

  The authors are grateful for helpful conversations with Boris Kruglikov, and acknowledge the use of the DifferentialGeometry package in Maple.  The research leading to these results has received funding from the Norwegian Financial Mechanism 2014-2021 (project registration number 2019/34/H/ST1/00636) and the Troms\o{} Research Foundation (project ``Pure Mathematics in Norway'').

\end{document}